\documentclass[12pt,a4paper]{article}
\usepackage[T1]{fontenc}
\usepackage{amsmath}
\usepackage{amssymb}
\usepackage{graphicx}
\usepackage{color}
\usepackage{amssymb}
\usepackage{latexsym}
\usepackage{dsfont,pifont}
\usepackage{bbold}

\renewcommand{\P}{{\rm P}}
\newcommand{\E}{{\rm E}}

\newcommand{\diag}{\mathrm{diag}}

\newcommand{\ud}{\, \mathrm{d}}  
\newcommand{\ue}{e} 
 
\newcommand{\indic}{\mathbb 1} 
\newcommand{\C}{\mathbb C} 

\newcommand{\EE}{\mathcal{E}}
\newcommand{\FF}{\mathcal{F}}
\newcommand{\HH}{\mathcal{H}}
\newcommand{\MM}{\mathcal{M}}
\newcommand{\cals}{\mathcal{S}}
\newcommand{\s}{\mathcal{S}}

\renewcommand{\u}{{\mbox{\tiny $+$}}}
\renewcommand{\d}{{\mbox{\tiny $-$}}}

\renewcommand{\u}{{\mbox{\tiny $+$}}}
\renewcommand{\d}{{\mbox{\tiny $-$}}}

\newcommand{\vect}[1]{\boldsymbol #1}
\newcommand{\vc}{\vect}
\newcommand{\vh}{\vect h}
\newcommand{\vxi}{\vect \xi}

\newcommand{\valpha}{\vect \alpha}

\newcommand{\vnu}{\vect \nu}
\newcommand{\vmu}{\vect \mu}
\newcommand{\vsigma}{\vect \sigma}
\newcommand{\vone}{\vect 1}
\newcommand{\vzero}{\vect 0}

\newcommand{\vcc}{\vect c}
\newcommand{\vcd}{\vect d}
\newcommand{\vm}{\vect m}

\newcommand{\vr}{\vect r}

\newcommand{\vligne}[1]{\begin{bmatrix} #1 \end{bmatrix}}

\newcommand{\Pid}{P_b}
\newcommand{\Piu}{L_b}
\newcommand{\hPid}{\widehat P_b}
\newcommand{\hPiu}{\widehat L_b}

\newtheorem{defn}{Definition}[section]
\newtheorem{lem}[defn]{Lemma}

\newtheorem{theo}[defn]{Theorem}
\newtheorem{thm}[defn]{Theorem}
\newtheorem{cor}[defn]{Corollary}
\newtheorem{rem}[defn]{Remark}
\newtheorem{ass}[defn]{Assumption}
\newtheorem{exemple}[defn]{Example}

\newcommand{\qed}{\hfill $\square$}
\newcommand{\bs}{\boldsymbol}
\newenvironment{proof}{
      \noindent {\bf Proof }}{\qed
      \vspace{0.25\baselineskip}
}
\newenvironment{proofbis}[1]{
      \noindent {\bf Proof of #1}}{\qed
      \vspace{0.25\baselineskip}
}
\newcommand{\debproof}{\begin{proof}}
\newcommand{\finproof}{\end{proof}}

\definecolor{darkmagenta}{rgb}{0.5,0,0.5}
\definecolor{darkgreen}{rgb}{0,0.6,0}
\definecolor{darkblue}{rgb}{0,0,0.6}
\definecolor{darkred}{rgb}{0.8,0,0}
\definecolor{mellow}{rgb}{.847, 0.72, 0.525}

\usepackage{pifont}

\begin{document}

\title{Feedback control: two-sided Markov-modulated Brownian motion with instantaneous change of phase at boundaries}
\author{Guy Latouche\thanks{Universit\'e Libre de Bruxelles,
D\'epartement d'informatique, CP 212, Boulevard du Triomphe, 1050 Bruxelles, Belgium, \texttt{latouche@ulb.ac.be}.} 
\and
Giang T. Nguyen\footnote{The University of Adelaide, School of Mathematical Sciences, SA 5005, Australia, \texttt{giang.nguyen@adelaide.edu.au}}
}
\date{\today}
\maketitle

\begin{abstract}
  We consider a Markov-modulated Brownian motion $\{Y(t), \rho(t)\}$
  with two boundaries at $0$ and $b > 0$, and allow for the
  controlling Markov chain $\{\rho(t)\}$ to instantaneously undergo a
  change of phase upon hitting either of the two boundaries at
  \emph{semi-regenerative epochs} defined to be the first time the process
  reaches a boundary since it last hits the other boundary. We call
  this process a \emph{flexible Markov-modulated Brownian motion}.

Using the recently-established links between stochastic fluid models
and Markov-modulated Brownian motions, we determine important
characteristics of first exit times of a Markov-modulated Brownian
motion from an interval with a regulated boundary.  These results allow us to
follow a Markov-regenerative approach and obtain the stationary
distribution of the flexible process.  This highlights the effectiveness of the regenerative approach in analyzing Markov-modulated Brownian motions subject to more general boundary behaviours than the classic regulated boundaries.\\

\noindent \underline{Keywords}: Fluid queues, Markov-modulated Brownian motion, regenerative processes, finite buffer, stationary distribution, feedback.
\end{abstract}

\section{Introduction}
	\label{s:introduction}
	
We analyze  Markov-modulated Brownian
motions (MMBMs) restricted to the interval $[0, b]$, the
distinguishing feature being that the processes are allowed to
undergo an instantaneous change of phase upon hitting either
boundaries at \emph{semi-regenerative epochs}. These
correspond to the first times the
process reaches a boundary after it hits the other boundary. We
refer to these processes as \emph{flexible  Markov-modulated
  Brownian motion}. 

The simplest example we have in mind is described as follows: consider
a buffer of finite capacity $b$ serving as temporary storage for data
in a communication network.  Assume that its content $X(t)$ evolves in
time like a Brownian motion with parameters $\mu<0$ and $\sigma^2 >
0$.  Whenever the buffer gets full, data may be lost.  To reduce
such losses, additional bandwidth is allocated, or the input stream is
throttled, or other measures are taken, such that the mean drift becomes $\mu' < \mu$.  Once the
buffer is emptied, the process returns to its normal mode of
operation, until it gets full again, etc.   See Figure~\ref{f:bmpath}
for two sample trajectories.  The graph at the top depicts a
regulated Brownian motion without change of parameters, the one at
bottom depicts a flexible version of the process; we have marked
with a thick line the interval of time $(0.46, 0.62)$ during which the drift is
$\mu'$.  

\begin{figure}[tbh]
\centering{
\includegraphics[width=0.8\textwidth,height=0.10\textheight]{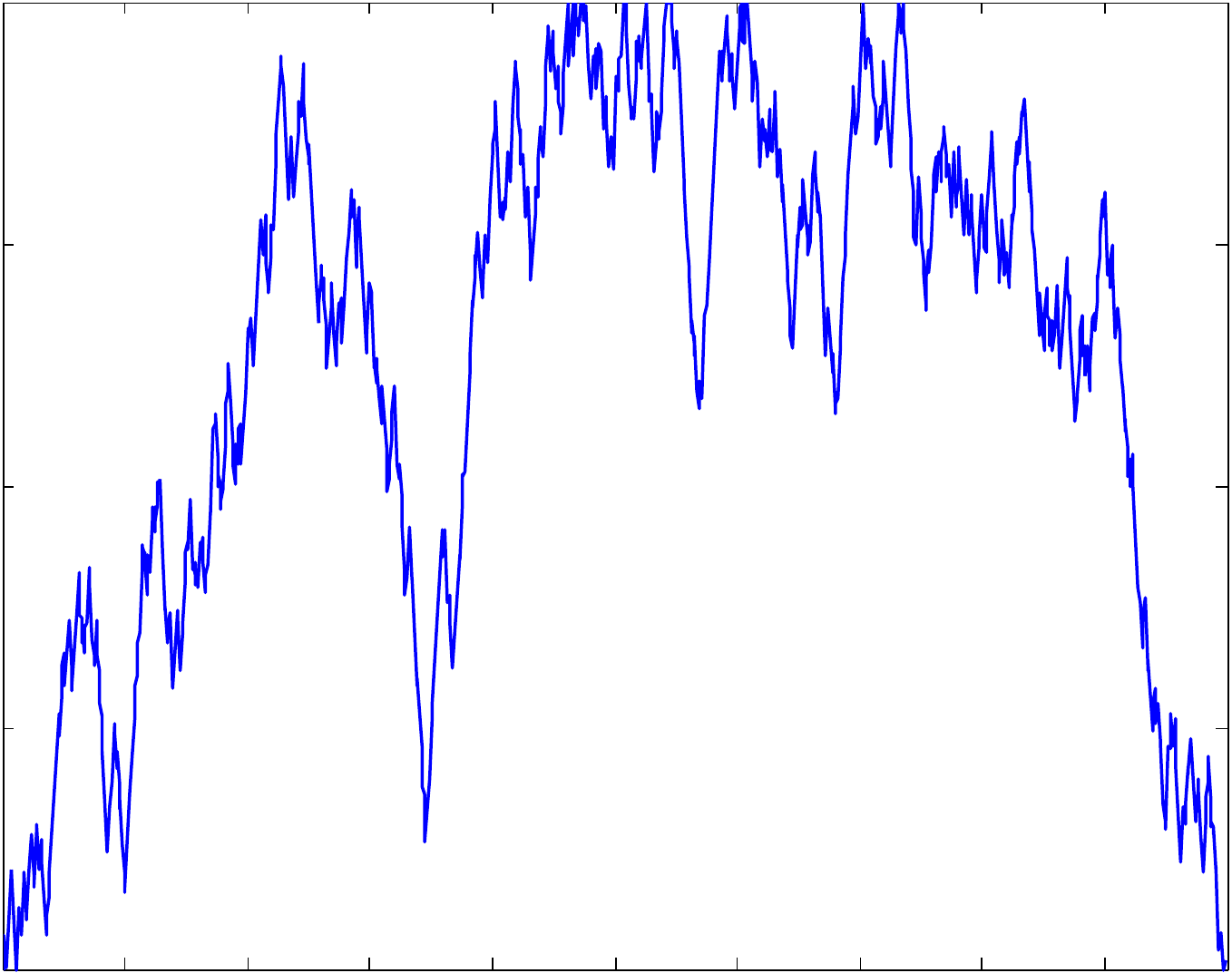} 
\put(0,-10){\makebox(0,0)[c]{$1$}}
\put(-155,-10){\makebox(0,0)[c]{$0.5$}}
\put(-310,-10){\makebox(0,0)[c]{$0$}}
\put(-320,0){\makebox(0,0)[c]{$0$}}
\put(-320,60){\makebox(0,0)[c]{$b$}}

\vspace{1\baselineskip}
  \includegraphics[width=0.8\textwidth,height=0.10\textheight]{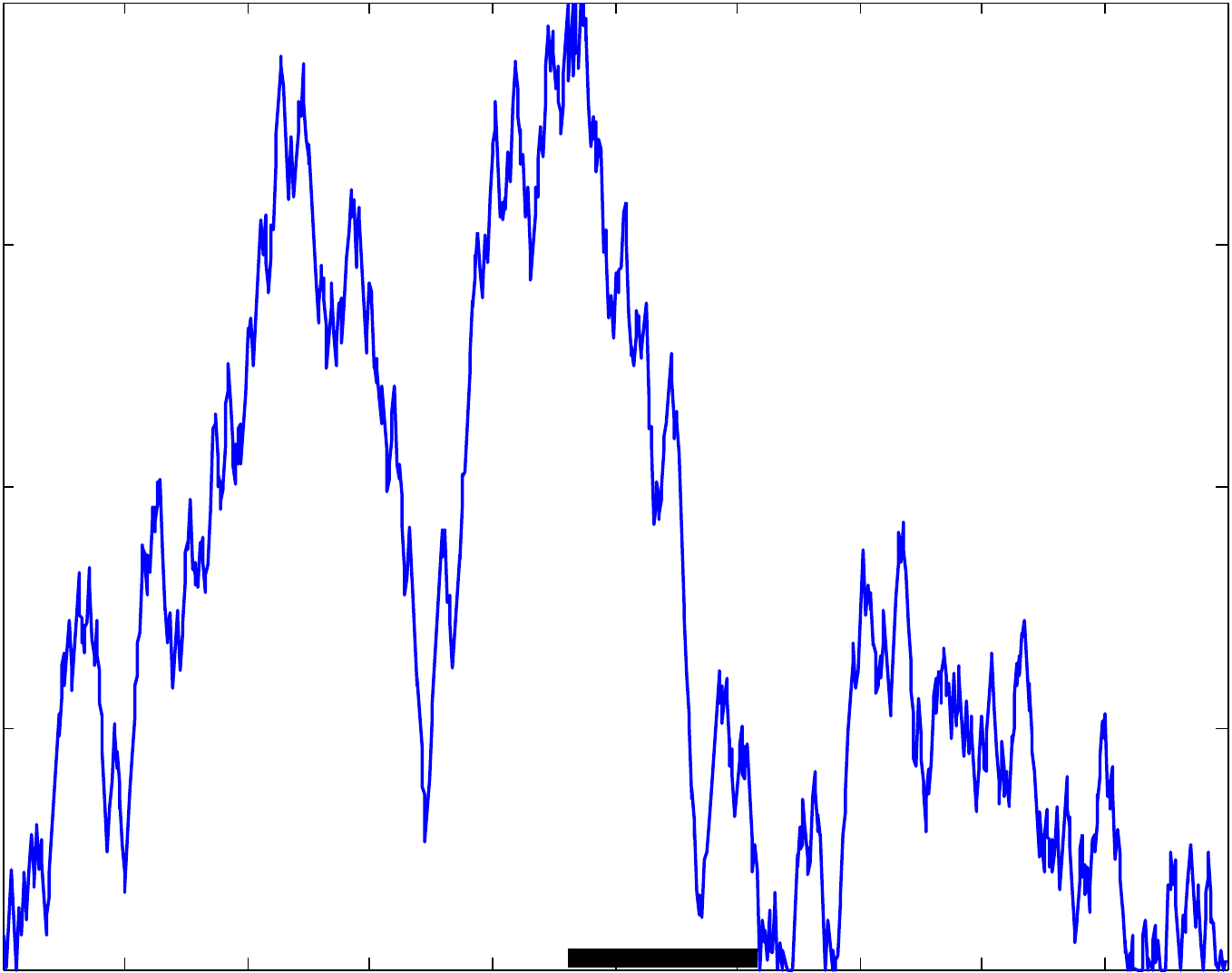} 
}
\put(0,-10){\makebox(0,0)[c]{$1$}}
\put(-155,-10){\makebox(0,0)[c]{$0.5$}}
\put(-310,-10){\makebox(0,0)[c]{$0$}}
\put(-320,0){\makebox(0,0)[c]{$0$}}
\put(-320,60){\makebox(0,0)[c]{$b$}}

\caption{\label{f:bmpath}
Sample trajectory of a regulated Brownian motion (top) and of a
flexible Brownian motion (bottom).  The parameters are $b=4$,
$\mu=-1$, $\mu'=-10$, $\sigma^2=10$.}
\end{figure}
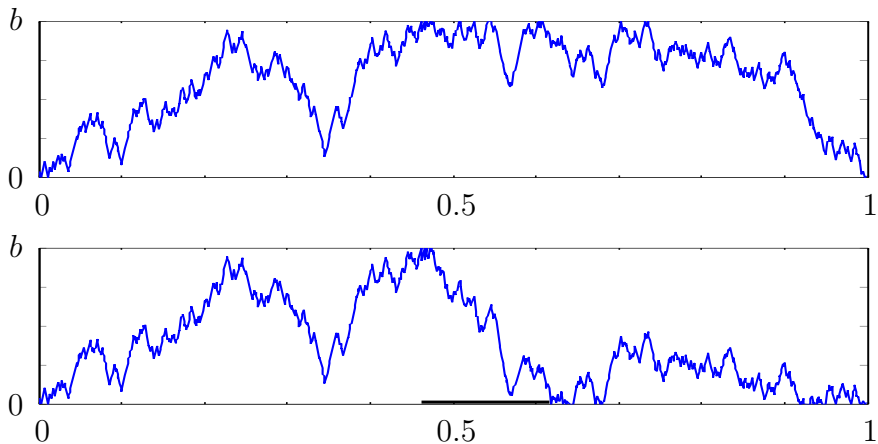





In the general formulation, the evolution of the buffer is controlled
by a continuous-time Markov chain called the process of phases.
Whenever the buffer reaches a boundary for the first time after it has
visited the other boundary, the phase is allowed to undergo an
instantaneous change.  It is clear that the epochs when change may
occur at a boundary form a semi-regenerative set of points.

The stationary distribution of a Markov-modulated Brownian motion
restricted to a strip $[0,b]$ is well-analyzed, as long as the
boundaries are absorbing or regulated as in Ivanovs~\cite{ivano10}.  Here, we determine the stationary
distribution of our flexible MMBMs by following a Markov-regenerative
approach.  The usefulness of this approach has been repeatedly
demonstrated for fluid queues (aka first-order fluid processes) as in
da Silva Soares and Latouche~\cite{dssl05,dssl07}, Latouche and
Taylor~\cite{lt09}, Bean and O'Reilly~\cite{bo08}.  It has been adapted
in Latouche and Nguyen~\cite{ln15} to MMBMs with one {\em reactive}
boundarys, by which we mean any boundary that is not absorbing or
regulated.  As we shall demonstrate later, the effectiveness of our method
extends far beyond the model analyzed here.

To determine the stationary distribution, the key ingredients 
needed are the expected time spent in an interval $[0,x]$ and in a
given phase during an excursion from 0 to $b$ for the process
regulated at 0, and from $b$ to 0 for the process regulated at $b$,
as well as the distribution of the phase upon reaching a boundary.
To obtain these quantities, we rely on the connections between MMBMs
and their approximating Markov-modulated fluid models, exemplified in
Latouche and Nguyen~\cite{ln14,ln13,ln15}.

Our results are related to those in Bean {\it et al.}~\cite{bot09}:
the authors analyze sojourn times in specified intervals during
various excursions for fluid queues with reactive boundaries.  The
main differences are that we consider jointly the level and the phase,
and that we deal with Markov modulated Brownian motion.  We should
also mention the results in Breuer~\cite{breue12} about occupation
times before a two-sided exit; we discuss later in some detail the
connection with our results.

The paper is organized as follows.  We give in the next section the
technical definition of flexible Markov-modulated Brownian motions and
we outline the Markov-regenerative approach for obtaining their
stationary distribution.  Section~\ref{s:fluid} includes background
material and notation required for the paper.
We determine in Sections~\ref{s:proba}
and~\ref{s:fluidtime} the first passage probabilities from one
boundary to the other, and the expected time spent during these
excursions.  In Section~\ref{s:sd}, we bring all partial results
together and determine the stationary distribution of a
flexible MMBM.  We give three numerical examples in
Section~\ref{s:illustration}, we compare in Section~\ref{s:lothar} our
results to the existing literature, and we conclude in
Section~\ref{s:conclusion} with a discussion on the applicability of
our approach to more complex models.

\section{Flexible Markov-modulated Brownian motion}
	\label{s:bm}
	

	A \emph{free-boundary} Markov-modulated Brownian motion is a
        two-dimensional Markov process $\{X(t), \kappa(t)\}_{t \geq
          0}$ such that
\begin{align*} 
		X(t) = X(0) + \int_0^t \mu_{\kappa(s)} \ud s + \int_0^t \sigma_{\kappa(s)}\ud B(s),  
\end{align*}
where $\{B(t)\}_{t \geq 0}$ is a standard Brownian motion and
$\{\kappa(t)\}$ is a continuous-time Markov chain on the state space
$\EE = \{1, \ldots, m\}$, with generator $Q$.  We denote by
$\Delta_\mu = \diag(\mu_i)_{i \in \EE}$ and $\Delta_\sigma^2 = \diag(\sigma_i^2)_{i \in
  \EE}$, respectively, the drift and variance matrices of the
MMBM.   We assume that $\sigma_i >0$ for all $i$ in order to
significantly simplify the presentation. 

A Markov-modulated Brownian motion $\{Z(t), \kappa(t)\}_{t \geq
  0}$ with two \emph{regulated} boundaries at $0$ and at $b > 0$ is
defined as
\begin{align*} 
			Z(t) = X(t) + R_0(t) - R_b(t),
\end{align*} 
where $\{R_0(t)\}_{t \geq 0}$ and $\{R_b(t)\}_{t \geq 0}$ are nonnegative,
continuous and almost surely nondecreasing processes; $R_0(t)$ increases only
when $Z(t)$ is at zero and $R_b(t)$ only when $Z(t)$ is at $b$. The
regulators $\{R_0(t)\}_{t \geq 0}$ and $\{R_b(t)\}_{t \geq 0}$ are the minimal processes keeping $Z(t)$ in $[0,b]$.

Next, we allow the phase to change as a reaction to $Z(t)$ reaching
either level $b$ or level 0.  The general idea is that the flexible
MMBM evolves like $Z(t)$ during a regeneration interval $(\theta,
\theta')$, but at time $\theta'$ the phase immediately switches to a new
value according to the transition probability matrix~$P^{\bullet}$, in
case $Z(\theta') = b$, and according to another probability matrix
$P^{\circ}$ in case $Z(\theta')=0$.  The technical details follow.

Our pathwise construction of a flexible MMBM $\{Y(t), \rho(t)\}_{t
  \geq 0}$ starts with a countably infinite supply of independent
copies of regulated MMBMs $\{Z_{n}(k;t), \kappa_{n}(k;t)\}$, such that
$\kappa_{n}(k;0) = k$, for $n \geq 0$ and $k \in \EE$.  We assume that
$Z_n(k;0)=0$ if $n$ is even, and $Z_n(k;0)=b$ for odd $n$.  All
processes have the same parameters: lower and upper bounds 0 and $b$,
generator $Q$, drift and variance matrices~$\Delta_\mu$ and
$\Delta_\sigma^2$.  Most of these processes will not be used, 
but they allow us to maintain independence where our construction
requires it.
	
Without loss of generality, assume that $Y(0)=0$ and $\rho(0)=i \in
\EE$, and let $\theta_0=0$.  We define
\[
Y(t) = Z_0(i;t), \qquad \rho(t)= \kappa_0(i;t) \qquad \qquad 
\mbox{ for $0 \leq t < \theta_1$,}
\]
where $\theta_1 = \inf \{t > 0: Z_{0}(i;t) = b\}$ is the first hitting
time to level $b$.   Upon hitting level $b$, the phase $\rho$
instantaneously changes to some value $j$ according to the transition
matrix $P^{\bullet}$, and we define
\[
Y(t)= Z_1(j;t-\theta_1), \qquad \rho(t)= \kappa_1(j;t-\theta_1)
\qquad  \mbox{ for  $\theta_1 \leq t < \theta_1+h_2$,}
\]
where  $h_2 = \inf \{ t > 0:
Z_1(j;t) = 0 \}$.  Upon hitting level 0, $\rho$ changes to a new
value according to $P^\circ$, and so on.

In general, starting with $\theta_0 =0$ and $\rho(0)=i$, we recursively define for $n \geq 0$ the following: 
\begin{itemize}
\item[]
$i_n = \rho(\theta_n)$,
\item[]
$h_{n+1} = \inf\{ t > 0: Z_n(i_n;t)= b \, \indic\{\mbox{$n$ is even}\}\}$,
\item[]
$\theta_{n+1} = \theta_n + h_{n+1}$, 
\item[]
$\rho(\theta_{n+1})$ is obtained from the row $\kappa_n(i_n; h_{n+1})$
of the matrix $P^\bullet$ if $n$ is even, and of $P^\circ$ if $n$ is odd,
\end{itemize}
and
\begin{equation}
   \label{e:yoft}
Y(t)= Z_N(\rho(\theta_N);t-S), \qquad \rho(t)=
\kappa_N(\rho(\theta_N);t-S), 
\end{equation}
where $N= \mathrm{arg\,max}\{s: \theta_s \leq t\}$ and $S=\theta_N$.
%
%
As $\lim_{n
\rightarrow \infty}\theta_n = \infty$, this defines $\{Y(t), \rho(t)\}$ over the whole
interval $[0, \infty)$.

By construction, 
\begin{align*} 
			\theta_{2n + 1} & = \inf \{t > \theta_{2n}:Y(t) = b\} \quad \mbox{ for } n \geq 0,  \\
			\theta_{2n} & = \inf \{t > \theta_{2n - 1}: Y(t) = 0\} \quad \mbox{ for } n \geq 1,
		\end{align*} 
and $\{\theta_n\}_{n \geq 0}$ forms a set of semi-regenerative points
for $\{Y(t), \rho(t)\}$ on the state space $\{0, b\} \times \EE$.
Because of the discontinuity introduced at the regeneration points, we
shall need to define the two limits, from the left and from the right,
\begin{align*} 
\rho(\theta_n^{-}) = \lim\limits_{t \uparrow \theta_n}\rho(t), 
\qquad
\rho(\theta_n^{+}) = \lim\limits_{t \downarrow \theta_n}\rho(t);
\end{align*} 
the trajectories of $\{\rho(t)\}$ are right-continuous, and so
$\rho(\theta_n) =	\rho(\theta_n^{+})$ for all~$n$.

The semi-Markov kernel $D(\cdot)$ for the transitions between
semi-regenerative points is defined as 
\begin{align}
   \label{e:Dij}
D_{(x,i), (y,j)}(t) &  = \P[\theta_{n + 1} - \theta_n \leq t,
\rho(\theta_{n + 1}^{-}) = j, Y(\theta_{n + 1}) = y \; | \\
   \nonumber
			& \hspace*{6cm} \; Y(\theta_n) = x, \rho(\theta_n^{-}) = i], 
\end{align} 	
where $x, y \in \{0, b\}$ and $i, j \in \EE$.  By
construction, again, the structure of $D$ is
\begin{align}
   \label{e:doft}
		D(t) = \left[\begin{array}{cc} 0 & D_{0}(t) \\ 
		            D_b(t) & 0    \end{array} \right],  
\end{align} 
where $D_x(\cdot)$ records the transition probabilities from
$\rho(\theta_n^-)$ to $\rho(\theta_{n + 1}^-)$ given that $Y(\theta_n) =
x$.

We need to make some irreducibility assumption at this point, for the
arguments that follow to hold.  In view of our elementary introductory
example, we should not automatically assume, as is usually done, that
the generator $Q$ is irreducible: in that example, one may consider
that there is one phase and two distinct sets of parameters, one for
the even-numbered and one for the odd-numbered intervals.  One may
also view that there are two phases, one with mean drift $\mu$ and one
with mean drift $\mu'$, and that there is no connection between the
two, except through the boundary feedback mechanism.  For the time
being, we make the following assumption only and we give in
Section~\ref{s:sd} simple conditions for it to hold.
\begin{ass}
   \label{a:irreducible}
The transition matrix $D$ is irreducible.  In other words, for any
pair of states $(x,i)$ and $(y,j)$, there is a path of positive probability
from $(Y(0)=x, \rho(0)=i)$ to $(Y(\theta_n)= y, \rho(\theta_n)=j)$, for
some $n$.
\end{ass}

Next, we define the matrix $\Theta(x)$ of
expected sojourn times: for $y =0$ or $b$, $i$ and $j$ in $\EE$, and $x\geq
0$, the component $\Theta_ {y;i,j}(x)$ is the expected time spent by
the process in the set $[0,x] \times \{j\}$ during a regeneration
interval $[\theta_n, \theta_{n+1})$, conditionally given that
$Y(\theta_n)=y$ and $\rho(\theta_n^-) =i$.  We display that matrix as
\begin{equation*}
\Theta(x) = \left[\begin{array}{c} 
		\Theta_{0}(x)  \\ 
		\Theta_{b}(x)
	\end{array} \right]. 
\end{equation*} 
By \c{C}inlar \cite[Sect.10.4, Prop.4.9]{cinla75}, the joint stationary distribution $\Pi(x)$ of $\{Y(t), \rho(t)\}$ is given by 
	\begin{align} 
		\label{eqn:Pix} 
	\Pi(x) = (\vcd \, \bs{\theta})^{-1} \vcd \, \Theta(x),
	\end{align} 
        where $\bs{\theta} = \Theta(\infty)\bs{1}$, with $\bs{1}$ a
        column vector of 1s, is the vector of expected length of a
        regenerative interval, given the initial state, and $\vcd$ is
        the stationary distribution of the phase immediately before
        the end of the next interval, that is, $\vcd D(\infty) =
        \vcd$, and $\vcd\bs{1} = 1$.

In Theorem~\ref{t:piofx} below, we express $\Pi(x)$ directly in terms
of properties of the regulated Brownian motion $\{Z(t), \kappa(t)\}$.
For that purpose, we introduce the transition probability matrices
\begin{align}
   \label{e:H0}
(H_0)_{ij} & = \P[\delta_b < \infty, \kappa(\delta_b) = j | Z(0)=0,
\kappa(0)=i], \\
   \label{e:Hb}
(H_b)_{ij} & = \P[\delta_0 < \infty, \kappa(\delta_0) = j | Z(0)=b,
\kappa(0)=i], 
\end{align}
where  $\delta_x = \inf\{ t >  0: Z(t)=x\}$  is the first passage time to level $x$, 
and we define the matrices of expected sojourn times  in
the interval $[0,x]$ during an excursion from level 0
to level $b$, and from level $b$ to level 0:
\begin{align}
   \label{e:M0}
(M_0(x))_{ij} &= \E[\int_0^{\delta_b} \indic\{Z(s) \in [0,x], \kappa(s)=j\}
\, \ud s  \, | Z(0)=0, \kappa(0)=i],  \\
   \label{e:Mb}
(M_b(x))_{ij} &= \E[\int_0^{\delta_0} \indic\{Z(s) \in [0,x], \kappa(s)=j\}
\, \ud s  \, | Z(0)=b, \kappa(0)=i].
\end{align}

\begin{thm}
   \label{t:piofx}
The stationary distribution $\Pi(x)$ of the flexible Markov-modulated
Brownian motion $\{Y(t),
\rho(t)\}$ is given by 
$
\Pi(x) = (\vnu \,   \vm)^{-1} \vnu M(x),
$
where  $\vnu = \vligne{\vnu_0 & \vnu_b}$ with 
\[
\vnu_0 = \vnu_0   H_0 P^\bullet H_b P^\circ  \qquad \vnu_b = \vnu_0
H_0 P^\bullet,
\]
$\vnu_0$ being unique up to a multiplicative constant, 
\[
M(x)   =\vligne{M_0(x) \\  M_b(x)},
\]
and $\vm = M(b) \vone$.
\end{thm}

\begin{proof}
By our definition (\ref{e:Dij}) of the semi-Markov kernel, starting
from one of the two boundaries at
time $\theta_n$, a new phase is chosen
with the corresponding matrix $P^\circ$ or $P^\bullet$, and then a
process stochastically identical to $\{Z(t)\}$ evolves until it
reaches the other boundary.  Thus,  $D_0(\infty) = P^\circ H_0$ and
$D_b(\infty) = P^\bullet H_b$ or, in matrix form,
\begin{equation}
   \label{e:dinfty}
D(\infty) = \vligne{P^\circ & 0 \\ 0 & P^\bullet}
           \vligne{0 & H_0 \\ H_b  & 0}.
\end{equation}
The stationary probability vector $\vcd$ of $D(\infty)$, written as
$\vcd = \vligne{\vcd_0 & \vcd_b}$,  satisfies the equations
\[
\vcd_0 = \vcd_b P^\bullet H_b, \qquad \vcd_b = \vcd_0 P^\circ H_0,
\]
or
\begin{equation}
   \label{e:dzerodb}
\vcd_0 = \vcd_0 P^\circ H_0 P^\bullet H_b, \qquad \vcd_b = \vcd_0
P^\circ H_0.
\end{equation}
The matrix $P^\circ H_0 P^\bullet H_b$ is the transition probability
matrix from a phase immediately before a regeneration at level 0 to
the phase immediately before the next regeneration at 0; it is
irreducible by Assumption~\ref{a:irreducible} and so the vector
$\vcd_0$ is unique, up to a multiplicative constant.

By the same argument that leads to (\ref{e:dinfty}), we conclude that 
\begin{equation*}
\vligne{\Theta_0(x) \\ \Theta_b(x)} = \vligne{P^\circ & 0 \\ 0 & P^\bullet}
           \vligne{ M_0(x) \\ M_b(x)},
\end{equation*}
and so, by (\ref{eqn:Pix}),
\[
\Pi(x) = c
           \vligne{\vcd_0 & \vcd_b} 
           \vligne{P^\circ & 0 \\ 0 & P^\bullet}
           \vligne{ M_0(x) \\ M_b(x)}
\]
for some normalizing constant $c$.
The remainder of the proof is immediate once we define $\vnu_0 =
\vcd_0 P^\circ$ and $\vnu_b = \vcd_b \P^\bullet$.
\end{proof}

In consequence of Theorem \ref{t:piofx}, we need only to take into
consideration a simple MMBM with two regulated boundaries, and to
focus on one excursion from a boundary to the other.  This we do in
Sections~\ref{s:proba} and~\ref{s:fluidtime}.  Before that, we recall
some basic properties of MMBMs.


\section{Background material and notation}
\label{s:fluid}

We analyze in Sections~\ref{s:proba} and~\ref{s:fluidtime} a regulated
process controlled by
a phase process with an irreducible generator $Q$ and a unique set of
parameters $\vmu$ and $\vsigma$.  Unlike the process defined in
Section~\ref{s:bm}, there is no reaction of the phase when the buffer
reaches a boundary and, to ensure that there is no confusion
with the matrices $H$ and $M$ defined in (\ref{e:H0}--\ref{e:Mb}), we use the
symbols $\HH$ and $\MM$.

Two matrices play an important role in the analysis of MMBMs.  They
are denoted as $U$ and $\widehat U$ in the present paper, and  $U$
and $-\widehat U$ are solutions of the matrix equation
\begin{equation}
   \label{e:U}
\Delta_\sigma^2 X^2 + 2 \Delta_\mu X + 2 Q =0,
\end{equation}
with $U$ being the minimal solution and $-\widehat U$ being the maximal
solution, meaning that the
eigenvalues of $U$ are the roots of the polynomial $\det(
\Delta_\sigma^2 z^2 + 2 \Delta_\mu z + 2 Q)$ in the negative
half complex plane and the eigenvalues of $-\widehat U$ are the roots
in the positive half-plane  (see D'Auria {\it et al.}~\cite[Section 4]{aikm12} and
Latouche and Nguyen~\cite[Lemma 5.3]{ln14}).

Both matrices are irreducible generators: $\widehat U$ is the generator of
the Markov chain $\{\kappa(\delta_x)\}_{x \geq 0}$ and $U$ is the
generator of the Markov chain $\{\kappa(\delta_{-x})\}_{x \geq 0}$.
One  recognizes three
different cases, based on the sign of the mean drift $\valpha
\vmu$, where $\valpha$ is the stationary distribution of
$Q$ ($\valpha Q = \vzero$, $\valpha \vone = 1$):
\begin{enumerate}
\item 
If $\valpha \vmu > 0$, then $\widehat U \vone = \vzero$ and $U
\vone \leq \vzero$,  with at least one strict inequality, $U$ is nonsingular.
\item 
If $\valpha \vmu < 0$, $U \vone = \vzero$ and $\widehat U
\vone \leq \vzero$,  with at least one strict inequality, $\widehat U$ is nonsingular.
\item 
If $\valpha \vmu = 0$, both $U \vone$ and $\widehat U \vone$
are equal to $\vzero$.
\end{enumerate}

We prove in~\cite{ln14, ln13}  that
Markov-modulated Brownian motions can be approximated by a family of
fast oscillating free-boundary fluid processes $\{X^{\lambda}(t),
\beta^{\lambda}(t), \kappa^{\lambda}(t)\}_{t \geq 0}$ parameterized by
$\lambda > 0$, where $\{ \beta^{\lambda}(t), \kappa^{\lambda}(t)\}$ is
a two-dimensional Markov process on the state space $\mathcal{S} =
\{1,2\} \times \EE$, with generator
\begin{align}
   \label{e:Tlam}
		T^{\lambda} = \left[\begin{array}{cc} Q - \lambda I & \lambda I \\ 
		\lambda I & Q - \lambda I \end{array} \right], 
\end{align} 
	and the level process $\{X^{\lambda}(t)\}$ is driven by the phase $\{\beta^{\lambda}(t), \kappa^{\lambda}(t)\}$ as follows
\begin{align*}
		X^{\lambda}(t) = \int_0^t C^{\lambda}_{\beta^{\lambda}(u), \kappa^{\lambda}(u)} \ud u,
\end{align*} 
with 
\begin{align}
   \label{e:Clam} 
C^{\lambda} = \left[\begin{array}{cc}
	\Delta_\mu + \sqrt{\lambda} \Delta_\sigma & \\
	& \Delta_\mu - \sqrt{\lambda} \Delta_\sigma
\end{array} \right].
\end{align}

Consider the family of regulated processes $\{Z^{\lambda}(t),
\beta^{\lambda}(t), \kappa^{\lambda}(t)\}$ with boundaries at $0$ and
at $b > 0$, and initial phase $\beta^\lambda(0)$ equal to 1 or 2
with equal probability 0.5, for all $\lambda$.
By \cite[Thm.3.1]{ln14}, the regulated MMBM $\{Z(t),
\kappa(t)\}$ defined in Section~\ref{s:bm} is the weak limit
of the projected process $\{Z^{\lambda}(t), \kappa^{\lambda}(t)\}$,
and the stationary distribution of the former arises as the limit of
that of the latter as $\lambda \rightarrow \infty$. 
In consequence, the sojourn time matrix $\MM(x)$ and the first passage
probability matrix $\HH$ are the limits of the corresponding matrices
for the projected process as $\lambda \rightarrow \infty$. 

We partition the state space into the subsets $\mathcal{S}_\u$ and 
$\mathcal{S}_\d$, where $\mathcal{S}_\u = \{(i,j) \in \mathcal{S}:
c_{ij} > 0\}$ and $\mathcal{S}_\d = \{(i,j) \in \mathcal{S}: c_{ij} <
0\}$. For sufficiently large~$\lambda$, $\mathcal{S}_\u = \{(1, j): j
\in \EE\}$ and $\mathcal{S}_\d = \{(2,j): j \in
\EE\}$.   Several matrices are partitioned in a conformant
manner.  For instance, we write $T^\lambda$ as
\[
T^\lambda = \vligne{
T^\lambda_{\u\u} & T^\lambda_{\u\d}
  \\
T^\lambda_{\d\u} & T^\lambda_{\d\d}}.
\]
Two first passage probabilities are needed in the next section.  One is $\Psi^\lambda_b$,  indexed by $\cals_\u \times \cals_\d$, which records the probability that, starting from level~0 in a state of $\cals_\u$, $Z^\lambda$ returns to level~0 before reaching
level $b$; the other matrix, $\Lambda^\lambda_b$ is indexed by $\s_\u
\times \s_\u$ and  records the probability that level $b$ is
reached before a return to level~0:  
\begin{align}
   \label{e:psib}
(\Psi^\lambda_b)_{(1,i)(2,j)} & = \P[\delta^\lambda_0 <
\delta^\lambda_b, \kappa^\lambda(\delta^\lambda_0)=j  \, | \, 
Z^\lambda(0)=0, \beta^\lambda(0)=1, \kappa^\lambda(0)=i],  \\
  \nonumber
(\Lambda^\lambda_b)_{(1,i)(2,j)} & = \P[\delta^\lambda_b <
\delta^\lambda_0, \kappa^\lambda(\delta^\lambda_b)=j  \, | \, 
Z^\lambda(0)=0, \beta^\lambda(0)=1, \kappa^\lambda(0)=i].
\end{align}


\section{Transition probability matrices}
\label{s:proba}

Starting from any state in $\s$ at time 0, the process $\{ Z^\lambda(t),
\beta^\lambda(t), \kappa^\lambda(t)\}$ is necessarily in a state of $\s_\u$
at time $\delta^\lambda_b$, and the matrix $H^\lambda_0$ of first
passage probability from level 0 to level $b$ has the structure
\[
H^\lambda_0 = \vligne{H^\lambda_{\u\u}  \\  H^\lambda_{\d\u}}.
\]
(We omit the subscript 0 for the sub-matrices in the following
calculations as there is no ambiguity.)  Since
$\beta^\lambda(0)$ is equal to 1 or 2 with equal probabilities, we
have 
\begin{equation}
   \label{e:Hlam1}
H_0 = \lim_{\lambda \rightarrow \infty} (0.5 H^\lambda_{\u\u} + 0.5
H^\lambda_{\d\u}). 
\end{equation}
Now, starting in a phase of $\s_\d$, the fluid remains at level 0
until it first moves to a phase of $\s_\u$, with the transition
probability matrix
\[
(-T^\lambda_{\d\d})^{-1} T^\lambda_{\d\u} = (\lambda I - Q)^{-1}
\lambda I = (I-\frac{1}{\lambda} Q)^{-1},
\]
so that $H^\lambda_{\d\u} = (I-\frac{1}{\lambda}
Q)^{-1} H^\lambda_{\u\u}$, and we see from (\ref{e:Hlam1}) that
\begin{equation}
   \label{e:Hlam2}
\HH_0 = \lim_{\lambda \rightarrow \infty }  H^\lambda_{\u\u}.
\end{equation}
Starting from level 0 in a phase of $\s_\u$, the fluid queue may move
directly to level $b$ without returning to level 0, or it may return
to level 0 before having reached level $b$.  Thus,
\begin{align*}
H^\lambda_{\u\u}  & = \Lambda^\lambda_b + \Psi^\lambda_b
H^\lambda_{\d\u}  \\
  & = \Lambda^\lambda_b + \Psi^\lambda_b  (I -\frac{1}{\lambda} Q)^{-1}
H^\lambda_{\u\u}   \\
 & = (I - \Psi^\lambda_b  (I -\frac{1}{\lambda} Q)^{-1})^{-1 }  \Lambda^\lambda_b 
\end{align*}
and we find that
\begin{equation}
   \label{e:limhpp}
\HH_0 = \lim_{\lambda \rightarrow \infty} (I - \Psi^\lambda_b )^{-1 }
\Lambda^\lambda_b .
\end{equation}
With this, we are in position to prove the following theorem.

\begin{thm}
   \label{t:probaup}
Consider an MMBM regulated at level 0.  The matrix $\HH_0$ of first
passage probability from level 0 to level $b$ is 
\[
\HH_0 = (-\Pid)^{-1} \Piu, 
\]
where $\Piu$ and $\Pid$ are solutions of the linear system
\begin{equation}
   \label{e:LoPo}
\vligne{\Piu & \Pid}
\vligne{I  &  e^{Ub} \\  e^{\widehat U b } & I}
=
\Delta_\sigma  \vligne{- \widehat U e^{\widehat U b}  & U}.
\end{equation}
The matrix $\Pid$ is a sub-generator and is nonsingular.

If $\valpha \vmu \not= 0$, then the system (\ref{e:LoPo}) is
nonsingular and its solution may be written as
\begin{align}
   \label{e:Lo}
\Piu & = -\Delta_\sigma (U + \widehat U) e^{\widehat U b} 
   (I - e^{U b} e^{\widehat U b})^{-1}   \\
   \label{e:Po}
\Pid & = \Delta_\sigma (U + \widehat U    e^{\widehat U b}  e^{U b}) 
   (I - e^{\widehat U b} e^{ U b})^{-1}.
\end{align}
In that case, 
\[
\HH_0 = e^{\widehat Ub} + (e^{-Ub} - e^{\widehat Ub}) (U e^{-Ub} + \widehat U
e^{\widehat Ub})^{-1} \widehat U e^{\widehat Ub}.
\]
If $\valpha \vmu = 0$, then the system (\ref{e:LoPo}) is singular and
one needs the additional equation
\begin{equation}
   \label{e:LoPob}
\Piu (b \vone - Q^\# \vmu) - \Pid Q^\# \vmu = \vsigma
\end{equation}
to completely characterize $\Piu$ and $\Pid$, where $Q^\#$ is the
group inverse of $Q$.
\end{thm}

\begin{proof}
It results from \cite[Lemma 5.5]{ln14} that 
\[
\Lambda^\lambda_b = \frac{1}{\sqrt\lambda} \Piu + O(\frac{1}{\lambda}),
\qquad 
\Psi^\lambda_b = I + \frac{1}{\sqrt\lambda} \Pid + O(\frac{1}{\lambda}),
\]
where ($\Piu, \Pid$) is a solution of (\ref{e:LoPo}) and that
$\Pid$ is nonsingular.   We readily conclude from (\ref{e:limhpp})
that $\HH_0 = (-\Pid)^{-1} \Piu$.   Furthermore, (\ref{e:Lo},
\ref{e:Po}) directly result from \cite[Eqn (33) and (34)]{ln14} when
$\valpha \vmu \not=0$.

If $\valpha \vmu = 0$, then both $e^{Ub}$ and $e^{\widehat U b}$ are
stochastic matrices, the coefficient matrix of (\ref{e:LoPo}) is
singular, and we need an additional equation.   D'Auria {\it et al.}
\cite{aikm10b} analyze first exit probabilities for the MMBM process
$\{X(t), \kappa(t)\}$ and determine the exit probabilities from the
interval $[0, b]$
\begin{align*}
P(x,0) & = \P[ \delta_0 < \delta_b, \kappa(\delta_0) | X(0)=x],
\\
P(x,b) & = \P[ \delta_b < \delta_0, \kappa(\delta_b) | X(0)=x],
\end{align*}
for $0 \leq x \leq b$.
Equations (56, 58) in \cite{aikm10b} may be written as 
\begin{equation}
   \label{e:56}
\vligne{P(x,b) & P(x,0)} 
\vligne{I  &  e^{Ub} \\  e^{\widehat U b } & I}
=
 \vligne{ e^{\widehat U (b-x)}  & e^{U x}}
\end{equation}
and 
\begin{equation}
   \label{e:58}
P(x,b) ((b-x) \vone + \vh) + P(x,0) (-x \vone + \vh) = \vh,
\end{equation}
with $\vh$ being any solution of the system $Q \vh = -\vmu$.

The matrix $Q$ has one eigenvalue equal to zero and such
solutions are of the form $\vh = - Q^\# \vmu + c
\vone$, where $c$ is an arbitrary scalar, and $Q^\#$ is the unique
solution of the linear system $X Q = I - \vone \valpha$, $ X \vone
=\vzero$; that matrix is called the group inverse of $Q$ (Campbell and
Meyer~\cite{cm91}).  As $Q$ is a generator, $Q^\#$ is also called the
deviation matrix of the Markov process with generator $Q$
(Coolen-Schrijner and van Doorn~\cite{cv02}), and one has $Q^\#=
\int_0^\infty (e^{Qu} -\vone\valpha) \ud u$.

In addition, it is shown in
\cite[Section 6.2]{ln14} that
\[
\Piu = \Delta_\sigma \lim_{x \rightarrow 0} \frac{\partial}{\partial
  x} P(x,b)
\qquad \mbox{and} \qquad
\Pid = \Delta_\sigma \lim_{x \rightarrow 0} \frac{\partial}{\partial
  x} P(x,0).
\]
Premultiplying both sides of (\ref{e:56}) by $\Delta_\sigma$ and taking
the derivative, we obtain(\ref{e:LoPo})  as $x \rightarrow 0$.
Similarly, 
\[
\Piu (b \vone + \vh)  - \Delta_\sigma P(0,b) \vone + \Pid \vh - \Delta_\sigma P(0,0) \vone = \vzero,
\]
follows from (\ref{e:58}).  As $P(0,b)=0$ and $P(0,0)=I$, the last equation  is identical to
(\ref{e:LoPob}) if we chose $\vh = -Q^\# \vmu$.  
This completes the proof.
\end{proof}

We may follow a similar line of argument to determine the matrix
$\HH_b$ of first passage probabilities from the upper boundary to the
boundary at level 0, We may also, as an alternative, define the
level-reversed process $\{\widehat X(t), \kappa(t)\}$, where
$\widehat X(t) = - Z(t)$.  For this process, the fluid rate vector
becomes $\widehat \vmu = - \vmu$, the r\^oles of the matrices $U$ and
$\widehat U$ are exchanged, and the first passage probability matrix
$\widehat{\HH}_0$
from 0 to $b$ of the regulated process of
$\{\widehat Z(t)\}$ is equal to $\HH_b$, the first passage probability
matrix of $\{Z(t)\}$ from $b$ to 0.  The proof of the corollary below
is immediate and is omitted.

\begin{cor}
   \label{t:probadown}
Consider an MMBM regulated at level 0 and $b$.  The matrix $\HH_b$ of
first passage probability from the boundary $b$ to the boundary 0 is
\[
\HH_b = (- \hPid)^{-1} \hPiu, 
\] where
\begin{equation}
   \label{e:Lhat}
\vligne{\hPid & \hPiu}
\vligne{I  &  e^{Ub} \\  e^{\widehat U b } & I}
=
\Delta_\sigma  \vligne{\widehat U &- U e^{ U b} }.
\end{equation}
The matrix $\hPid$ is an irreducible subgenerator and is nonsingular.

If $\valpha \vmu \not= 0$, then 
\begin{align*}
\hPiu & = -\Delta_\sigma (U + \widehat U) e^{ U b} 
   (I - e^{\widehat U b} e^{ U b})^{-1}   \\
\hPid & = \Delta_\sigma (\widehat U + U    e^{ U b}  e^{\widehat U b}) 
   (I - e^{ U b} e^{\widehat U b})^{-1},
\end{align*}
and
\[
\HH_b = e^{ Ub} + (e^{-\widehat Ub} - e^{ Ub}) (\widehat U e^{-\widehat Ub} + U
e^{ Ub})^{-1} U e^{ Ub}.
\]
If $\valpha \vmu = 0$, then $\hPiu$ and $\hPid$ are
determined by the system (\ref{e:Lhat}) and the additional equation 
\begin{equation}
   \label{e:LoPoc}
\hPiu (b \vone + Q^\# \vmu) + \hPid Q^\# \vmu = \vsigma.
\end{equation}
\qed
\end{cor}

\section{Expected time in $[0, x]$ during an excursion}
\label{s:fluidtime}

We determine in this section the matrix $\MM_0(x)$ of expected sojourn
time of a regulated MMBM $\{Z(t),\kappa(t)\}$ during an excursion from 0 to
$b$.  It soon becomes clear that to do so, we need to deal at
the same time with excursions from $b$ to 0 by the same process.  The
matrix of expected sojourn time in $[0,x]$ during such an excursion is
denoted as $\MM_b(x)$.  

We define $M_0^\lambda(x)$ to be the matrix of expected sojourn time
of the rapidly switching process $\{Z^\lambda(t), \beta^\lambda(t),
\kappa^\lambda(t) \}_{t \geq 0}$ in
$[0,x]$ during an excursion from 0 to $b$:
\begin{align}
   \nonumber
(M_0^\lambda(x))_{(\ell,i)(k,j)} &= \E[\int_0^{\theta^\lambda} \indic\{
Z^\lambda(s) \in [0,x], \beta^\lambda(s)=k, \kappa^\lambda(s)=j\}
\\
   \label{e:molam}
  & \qquad  \qquad | Z^\lambda(0) = 0, \beta^\lambda(0)=\ell, \kappa^\lambda(0)=i],
\end{align}
where $\theta^\lambda = \inf\{t > 0: Z^\lambda(t)=b\}$ is the first
passage time to level $b$.  We partition that matrix as 
\[
M_0^\lambda(x) = \vligne{M_{0;\u\u}^\lambda(x) & M_{0;\u\d}^\lambda(x) \\
  M_{0;\d\u}^\lambda(x) & M_{0;\d\d}^\lambda(x)}
\]
and, by an argument similar to the one that leads to (\ref{e:Hlam2}),
we find that 
$
\MM_0(x) = \lim_{\lambda \rightarrow \infty} (M_{0;\u\u}^\lambda(x) +
M_{0;\u\d}^\lambda(x) )$.

Next, we define $M_b^\lambda(x)$ to be the matrix of expected
sojourn in $[0,x]$ during an excursion from $b$ to 0 and we partition
it as 
\[
M_b^\lambda(x) = \vligne{M_{b;\u\u}^\lambda(x) & M_{b;\u\d}^\lambda(x) \\
  M_{b;\d\u}^\lambda(x) & M_{b;\d\d}^\lambda(x)};
\]
one shows that $\MM_b(x) = \lim_{\lambda \rightarrow \infty} (M_{b;\d\u}^\lambda(x) +
M_{b;\d\d}^\lambda(x) )$.  To simplify our equations in the remainder
of this section, we write
\[
M_{0;\u}^\lambda = \vligne{M_{0;\u\u}^\lambda(x) &
  M_{0;\u\d}^\lambda(x)} 
\qquad \mbox{and} \qquad
M_{b;\d}^\lambda = \vligne{M_{b;\d\u}^\lambda(x) &M_{b;\d\d}^\lambda(x)},
\]
and we summarize as follows the discussion above:
\begin{equation}
   \label{e:mofx}
\vligne{\MM_0(x) \\ \MM_b(x)}
= \lim_{\lambda \rightarrow \infty}
\vligne{M_{0;\u}^\lambda(x)  \\ M_{b;\d}^\lambda(x)}  
\vligne{I \\ I}.
\end{equation}

\begin{thm}
   \label{t:expected}
If $\valpha \vmu \not= 0$, then
\begin{equation}
   \label{e:mofxb}
\vligne{\MM_0(x) \\ \MM_b(x)}
= 2 \vligne{-\Pid^{-1} & \\ & -\hPid^{-1}}
\vligne{I  &  e^{Kb} \\  e^{\widehat Kb}  &  I}^{-1}
\vligne{\FF(K;x) \\  e^{\widehat K(b-x)} \FF(\widehat K; x)} \Delta_\sigma^{-1},
\end{equation}
where
\begin{equation}
   \label{e:F}
\FF(A;x) = \int_0^x e^{Au} \, \ud u,
\end{equation}
and
\[
K  = \Delta_\sigma  U   \Delta_\sigma^{-1}  +  2 \Delta_\sigma^{-2}
\Delta_\mu    
\qquad \mbox{and}  \qquad
\widehat K  = \Delta_\sigma  \widehat U   \Delta_\sigma^{-1} -  2
\Delta_\sigma^{-2} \Delta_\mu.
\]
\end{thm}

To prove this, we 
we proceed in three preliminary steps: we express $M_{0;\u}^\lambda(x)$ and
$M_{b;\d}^\lambda(x)$ in terms of exit times from the interval
$(0,b)$, next we analyze first passage times for the unregulated
fluid process, and we establish a connection between the two.  In the
final step we prove (\ref{e:mofxb}) through a limiting argument.

\paragraph{Step A.}

We
define the matrix $N_0^\lambda(x)$ of sojourn time in $[0,x]$ until
$\{X^\lambda(t)\}$ hits {\em either} level 0 {\em or} level $b$, starting
from 0 in a phase of $\s_\u$:
\begin{align}
   \nonumber
(N_0^\lambda(x))_{(1,i);(k,j)} & = \E[\int_0^{\delta_0^\lambda \wedge
  \delta_b^\lambda} \indic\{ X^\lambda(s) \in [0,x],
\beta^\lambda(s)=k, \kappa^\lambda(s)=j\} \ud s
\\
   \label{e:nolam}
  & \qquad  \qquad | X^\lambda(0) = 0, \beta^\lambda(0)=1,
  \kappa^\lambda(0)=i],
\end{align}
for $(1,i) \in \s_\u$, $(k,j) \in \s$.

\begin{lem}
   \label{lem:0lamb}  The matrix $M_{0;\u}^\lambda(x)$ of expected
   sojourn time in $[0,x]$ during the excursion $[0, \theta^\lambda]$
   from level 0 to level $b$ is given by
\begin{align}
   \label{e:molamA}
M_{0;\u}^\lambda(x) = (I - \Psi_b^\lambda (-T_{\d\d}^\lambda)^{-1}
T_{\d\u}^\lambda)^{-1} (N_0^\lambda(x) + \Psi_b^\lambda 
\vligne{0 & (-T_{\d\d}^\lambda)^{-1}}),
\end{align}
where $\Psi_b^\lambda$ and $N_0^\lambda(x)$ are defined in
(\ref{e:psib}) and (\ref{e:nolam}), respectively. 
\end{lem}
\begin{proof}
We decompose the interval $[0, \theta^\lambda]$ as $[0,
\delta_0^\lambda \wedge \delta_b^\lambda] \cup [\delta_0^\lambda
\wedge \delta_b^\lambda , \theta^\lambda]$, and obtain
\begin{equation}
   \label{e:molamB}
M_{0;\u}^\lambda(x)  = N_0^\lambda(x) + \Psi_b^\lambda
\vligne{0 & (-T_{\d\d}^\lambda)^{-1}}   + \Psi_b^\lambda (-T_{\d\d}^\lambda)^{-1}
T_{\d\u}^\lambda  M_{0;\u}^\lambda(x).
\end{equation}
To justify this, we observe that the process must accumulate time in
$[0,x]$ until it hits one of the boundaries; this corresponds to the
first term in (\ref{e:molamB}).  With probability $\Psi_b^\lambda$,
the process has returned to level 0, where it accumulates more time (the
second term), and then leaves level 0 and accumulates time during the
remainder of the excursion (the third term).  Equation
(\ref{e:molamA}) immediately follows.
\end{proof}

The proof of the next lemma is omitted as it merely mimics the proof
of Lemma~\ref{lem:0lamb}.  For excursions that start in $b$, we
define a new set of matrices: the transition probability matrices
\begin{align}
   \label{e:psibhat}
\widehat\Psi^\lambda_b & = \P[\delta_b^\lambda <
\delta_0^\lambda, \kappa^\lambda(\delta^\lambda_b)  \, | \, 
Z^\lambda(0)=b, \beta^\lambda(0) = 2, \kappa^\lambda(0)] \qquad \mbox{on
  $\s_\d \times \s_\u$,}  \\
  \nonumber
\widehat\Lambda^\lambda_b & = \P[\delta_0^\lambda <
\delta_b^\lambda, \kappa^\lambda(\delta_0^\lambda)  \, | \, 
Z^\lambda(0)=b, \beta^\lambda(0) = 2, \kappa^\lambda(0)] \qquad \mbox{on
  $\s_\d \times \s_\d$,} 
\end{align}
and the matrix $N_b^\lambda(x)$ of sojourn time in $[0,x]$ until
$\{X^\lambda(t)\}$ hits level 0 or level~$b$, starting from level~$b$ in
a phase of $\s_\d$:
\begin{align}
   \nonumber
N_b^\lambda(x) & = \E[\int_0^{\delta_0^\lambda \wedge
  \delta_b^\lambda} \indic\{ X^\lambda(s) \in [0,x], \beta^\lambda(s),
\kappa^\lambda(s)\}   \ud s
\\
   \label{e:nblam}
  & \qquad  \qquad | X^\lambda(0) = b, \beta^\lambda(0)=2,
  \kappa^\lambda(0)]
\end{align}
on $\s_\d \times \s$.

\begin{lem}
   \label{t:blamb}  The matrix $M_{b;\d}^\lambda(x)$ of expected
   sojourn time in $[0,x]$ during an excursion from level $b$ to level
   0 is given by 
\begin{equation*}
M_{b;\d}^\lambda(x) = (I - \widehat\Psi_b^\lambda (-T_{\u\u}^\lambda)^{-1}
T_{\u\d}^\lambda)^{-1} (N_b^\lambda(x) + \widehat\Psi_b^\lambda
\vligne{0 & (-T_{\u\u}^\lambda)^{-1}}) 
\end{equation*}
where $\widehat\Psi_b^\lambda$ and $N_b^\lambda(x)$ are defined in
(\ref{e:psibhat}) and (\ref{e:nblam}), respectively. 
\qed
\end{lem}

\paragraph{Step B.}

Next, we characterize expected sojourn times during
intervals $(0, \delta_0^\lambda)$ or $(0, \delta_b^\lambda)$ for the
{\em unregulated} process $\{X^\lambda(t), \beta^\lambda(t),
\kappa^\lambda(t)\}$.  We define the matrices 
\begin{align*}
\Gamma_0^\lambda(x)) & = \E[\int_0^{\delta_0^\lambda} \indic\{
X^\lambda(s) \in [0,x], \beta^\lambda(s), \kappa^\lambda(s)\} \ud s
\\
  & \qquad  \qquad | X^\lambda(0) = 0, \beta^\lambda(0)=1,
  \kappa^\lambda(0)],
\end{align*}
indexed by $\s_\u \times \s$, and 
\begin{align*}
\widehat\Gamma_b^\lambda(x) & = \E[\int_0^{\delta_b^\lambda} \indic\{
X^\lambda(s) \in [0,x], \beta^\lambda(s), \kappa^\lambda(s)\} \ud s
\\
  & \qquad  \qquad | X^\lambda(0) = b, \beta^\lambda(0)=2,
  \kappa^\lambda(0)],
\end{align*}
indexed by $\s_\d \times \s$.  The matrix 
$\Gamma_0^\lambda(x)$ records the expected sojourn time of
$\{X^\lambda(t), \beta^\lambda(t),
\kappa^\lambda(t)\}$ in the interval $[0,x]$ during an interval of first return 
to~0, starting from 0 in a phase of $\s_\u$, while
$\widehat\Gamma_b^\lambda(x)$ corresponds to a first return to $b$, starting
from level $b$ in a phase of $\s_\d$.

We can show that
\begin{align}
   \label{e:gamma0}
\Gamma_0^\lambda(x) & = \int_0^x  e^{K^\lambda u} \ud u
\vligne{(C_\u^\lambda)^{-1} & \Psi^\lambda |C_\d^\lambda|^{-1}}
\\   \nonumber
  & = \FF(K^\lambda;x)
\vligne{(C_\u^\lambda)^{-1} & \Psi^\lambda |C_\d^\lambda|^{-1}},
\end{align}
where 
\begin{equation}
   \label{e:K}
K^\lambda = (C_\u^\lambda)^{-1}  T_{\u\u}^\lambda  + \Psi^\lambda
|C_\d^\lambda|^{-1} T_{\d\u}^\lambda
\end{equation}
and $\Psi^\lambda$, indexed by $\s_\u \times \s_\d$, is the matrix of
first return probability from level 0 back to level 0; it is the
minimal nonnegative solution of the Riccati equation
\begin{equation}
   \label{e:riccati}
		(C_{\u}^\lambda)^{-1} T_{\u\d}^\lambda + (C_{\u}^\lambda)^{-1} T_{\u\u}^\lambda \Psi^\lambda + \Psi^\lambda |C_{\d}^\lambda|^{-1} T_{\d\d}^\lambda + \Psi^\lambda |C_{\d}^\lambda|^{-1} T_{\d\u}^\lambda\Psi^\lambda  = 0.
\end{equation}
For details, we refer to Rogers~\cite{roger94} and Latouche and
Nguyen~\cite{ln13}.  We give in Appendix~\ref{s:gamma0} a technical demonstration of
(\ref{e:gamma0}); a simple
justification is that $(e^{K^\lambda u})_{s, s'}$ is, for 
any states $s$ and $s'$ in $\s_\u$, the expected number of crossings of level $u$
in phase $s'$ under taboo of the level 0, given that the process
$\{X^\lambda(t)\}$ starts in level 0 and phase $s$
(Ramaswami~\cite{ram99}).

If $\valpha\vmu < 0$, then all eigenvalues of $K^\lambda$ are in
$\C_{<0}$, the set of complex numbers with strictly negative real
part, and $K^\lambda$ is nonsingular.  If $\valpha\vmu \geq 0$, then
one eigenvalue is equal to 0, the others are in $C_{<0}$, and
$K^\lambda$ does not have an inverse.  Thus, the integral in
(\ref{e:gamma0}) takes different algebraic forms according to the
case.  

\begin{lem}
   \label{t:F}
If all the eigenvalues of the matrix $A$ are in $\C_{<0}$, then
\begin{equation}
   \label{e:FA}
\FF(A;x) = (-A)^{-1} (I - e^{Ax}).
\end{equation}
If $A$ has all its eigenvalues in $\C_{<0}$, with the exception of one
eigenvalue equal to 0, then
\begin{equation}
   \label{e:FAS}
\FF(A;x) = (-A^\#) (I - e^{Ax}) + x \vc v_a \vc u_a,
\end{equation}
where $\vc v_a$ and $\vc u_a$ are the right- and left-eigenvectors of
$A$ associated to the eigenvalue 0, and $A^\#$ is the group inverse of $A$.
\end{lem}
\begin{proof}
The proof is by verification that both sides
of (\ref{e:FA}) and of (\ref{e:FAS}) are equal for $x=0$ and have the
same derivative with respect to $x$.
\end{proof}

The matrix $\widehat \Gamma_b^\lambda(x)$ is given by 
\begin{align}
   \nonumber
\widehat\Gamma_b^\lambda(x) & = \int_{b-x}^b  e^{\widehat K^\lambda u} \ud u
\vligne{|C_\d^\lambda|^{-1}  & \widehat\Psi^\lambda
  (C_\u^\lambda)^{-1}}
 \\
   \label{e:gammab}
   & = e^{\widehat K^\lambda(b-x)} \FF(\widehat K^\lambda; x)
    \vligne{|C_\d^\lambda|^{-1} & \widehat\Psi^\lambda
      (C_\u^\lambda)^{-1}},
\end{align}
where 
\begin{equation}
   \label{e:Khat}
\widehat K^\lambda = |C_\d^\lambda|^{-1} T_{\d\d}^\lambda +  \widehat\Psi^\lambda
      (C_\u^\lambda)^{-1} T_{\u\d}^\lambda
\end{equation}
and $\widehat\Psi^\lambda$, indexed by $\s_\d \times \s$, is the
matrix of first return probability from level $b$ back to level $b$.
It is the minimal nonnegative solution of the equation
\begin{equation}
   \label{e:riccatihat}
|C_\d^\lambda|^{-1}  T_{\d\u}^\lambda  +
|C_\d^\lambda|^{-1}  T_{\d\d}^\lambda \widehat\Psi^\lambda +
\widehat\Psi^\lambda   (C_\u^\lambda)^{-1}  T_{\u\u}^\lambda +
\widehat\Psi^\lambda   (C_\u^\lambda)^{-1}  T_{\u\d}^\lambda
\widehat\Psi^\lambda =0.
\end{equation}
To prove (\ref{e:gammab}), we define the level-reversed process
$\{\widehat X^\lambda(t), \beta^\lambda(t), \kappa^\lambda(t)\}$ with fluid
rate vector $\widehat \vmu = - \vmu$, and we observe that the time
spent in $[0,x]$ by the process $\{ X^\lambda(t), \beta^\lambda(t),
\kappa^\lambda(t)\}$ during an interval of first return to $b$ has the same distribution as the time spent in
$[b-x, b]$ by the level-reversed process during an interval of first
return to 0.

We note for future reference that the eigenvalues of $\widehat
H^\lambda$ are in $\C_{<0}$ if $\valpha \vmu > 0$, otherwise the
matrix has one eigenvalue equal to 0, with the others in $\C_{<0}$.

\paragraph{Step C.}

In the third step, we establish a relation between $(N_0^\lambda(x),
N_b^\lambda(x))$ and $(\Gamma_0^\lambda(x),
\widehat\Gamma_b^\lambda(x))$, which leads us to an expression for the
matrices $M_{0;\u}^\lambda(x)$ and $M_{b;\d}^\lambda(x)$ as functions
of $K^\lambda$, $\widehat K^\lambda$, $\Psi^\lambda$, and
$\widehat \Psi^\lambda$.

\begin{lem} 
		 	\label{lem:N}
The matrix $N^{\lambda}(x)$ of mean sojourn times in $(0,x)$ during the
interval $(0, \delta_0^\lambda \wedge \delta_b^\lambda)$ is a
solution of the system  
\begin{align}
		 	\label{eqn:N-sol}
\vligne{ I & e^{K^\lambda b} \Psi^\lambda  \\ 
     e^{\widehat{K}^\lambda b}\widehat{\Psi}^\lambda  & I}
\vligne{N_0^{\lambda}(x)  \\ N_b^{\lambda}(x) }
= 
\vligne{\Gamma_0^{\lambda}(x) \\ \widehat{\Gamma}_b^{\lambda}(x) }.
\end{align}
where  $K^\lambda$ and $\widehat K^\lambda$ are given in (\ref{e:K})
and (\ref{e:Khat}).
\end{lem} 

\begin{proof} 
	The proof is similar to that of \cite[Lemma~4.1]{dssl03} and
        we give below its general outline only. First, observe that 
\begin{align*}
	\Gamma_0^{\lambda}(x) = N_0^{\lambda}(x) + \Lambda_b \Gamma^*(x),  
\end{align*} 
where $\Gamma^*(x)$ is the matrix of mean sojourn time in the interval
$(0,x)$ before the unregulated process $X^\lambda(t)$ first returns to
level $0$, starting from level $b$ in a phase of  $\s_\u$. Also, 
\begin{align*}
  \Gamma^*(x) & = \Psi^\lambda (N_b^{\lambda}(x) + \widehat \Psi_b
  \Gamma^*(x)) = (I - \Psi^\lambda \widehat \Psi_b)^{-1} \Psi^\lambda
  N_b^{\lambda}(x),
\end{align*}
and thus 
\begin{align*} 
	\Gamma_0^{\lambda}(x) = N_0^{\lambda}(x) + \Lambda_b   (I -
        \Psi^\lambda \widehat  \Psi_b)^{-1} \Psi^\lambda N_b^{\lambda}(x). 
\end{align*}
Now, we recognise that
	\begin{align*} 
          \Lambda_b (I - \Psi^\lambda \widehat \Psi_b)^{-1} =
          \Lambda_b (I + \Psi^\lambda \widehat \Psi_b + (\Psi^\lambda
          \widehat \Psi_b)^2 + (\Psi^\lambda \widehat \Psi_b)^3 +
          \cdots)
	\end{align*}
is the matrix of expected number of visits to level $b$ in a phase of
$\s_\u$, starting from $0$, before the first return to level 0,
and is thus equal to $\ue^{Kb}$ (Ramaswami~\cite{ram99}).  This gives the first equation in
(\ref{eqn:N-sol}); the second is similarly proved.
	\end{proof} 	
	
\begin{rem} \em
   \label{r:nonzero}
   By \cite[Lemma 4.2]{dssl03}, the coefficient matrix in
   (\ref{eqn:N-sol}) is nonsingular if $\valpha \vmu \not= 0$, and so
   Lemma~\ref{lem:N} completely characterizes $N^{\lambda}(x)$ for
   fluid processes with non-zero mean drift.
\end{rem}



\begin{proofbis}{Theorem~\ref{t:expected}} 
By \cite[Lemmas 5.1 and 5.2]{ln14}, we have
\begin{align*}
\Psi^\lambda & = I + O({1}/{\sqrt\lambda}),
&\widehat\Psi^\lambda & = I + O({1}/{\sqrt\lambda}),
\\
K^\lambda & = K + O({1}/{\sqrt\lambda}),
&\widehat K^\lambda & = \widehat K + O({1}/{\sqrt\lambda})
\end{align*}
and so, by (\ref{e:gamma0}, \ref{e:gammab}), 
\[
\vligne{\Gamma_0^\lambda(x) \\ \widehat\Gamma_b^\lambda(x)}
=
\frac{1}{\sqrt\lambda}  \vligne{\FF(K;x) &  \\  & e^{\widehat K(b-x)} \FF(\widehat
  K;x)}
\vligne{\Delta_\sigma^{-1} & \\ & \Delta_\sigma^{-1}} \vligne{I & I \\
  I & I}   
+O(\frac{1}{\lambda})
\]
since both $C_\u^\lambda = |C_\d^\lambda| = \sqrt\lambda \Delta_\sigma + O(1)$.  
Therefore, (\ref{eqn:N-sol}) becomes
\begin{align}
  \nonumber
&\vligne{I & e^{K b} \\ e^{\widehat K b} & I} 
\vligne{N_0^\lambda(x) \\ N_b^\lambda(x)} \\
   \label{e:p4}
& \qquad = 
\frac{1}{\sqrt\lambda}  \vligne{\FF(K;x) \Delta_\sigma^{-1} &  \\  & e^{\widehat K(b-x)} \FF(\widehat K;x) \Delta_\sigma^{-1} }
\vligne{I & I \\ I & I} +O(\frac{1}{\lambda}).
\end{align}
By \cite[Lemma 3.6]{ln13},
\begin{itemize}
\item 
if $\valpha\vmu < 0$, then $K$ has all eigenvalues in $\C_{<0}$ and
$\widehat K$ has $m-1$ eigenvalues in $\C_{<0}$ and one eigenvalue
equal to 0,
\item 
if $\valpha\vmu > 0$, then $\widehat K$ has all eigenvalues in $\C_{<0}$ and
$K$ has $m-1$ eigenvalues in $\C_{<0}$ and one eigenvalue
equal to 0.
\end{itemize}
This entails that one of the two matrices $e^{Kb}$ and $e^{\widehat K
  b}$ has spectral radius equal to one, with the other having spectral
radius strictly less than one, so that the left-most matrix in (\ref{e:p4})
is nonsingular when $\valpha\vmu \not=0$, and 
\begin{align}
  \nonumber
\vligne{N_0^\lambda(x) \\ N_b^\lambda(x)}
 =  &
\frac{1}{\sqrt\lambda} 
\vligne{I & e^{K b} \\ e^{\widehat K b} & I}^{-1}
\\
   \label{e:p1}
& \vligne{\FF(K;x) \Delta_\sigma^{-1} &  \\  & e^{\widehat K(b-x)} \FF(\widehat
  K;x) \Delta_\sigma^{-1} }
\vligne{I & I \\ I & I} +O(\frac{1}{\lambda}).
\end{align}
On the other hand, we have $(I - \Psi_b ^\lambda)^{-1} = \sqrt\lambda
(-\Pid)^{-1} +O(1)$ by \cite[Lemma 5.5]{ln14} and
$(T_{\d\d}^\lambda)^{-1} = O(1/\lambda)$, $(-T_{\d\d}^\lambda)^{-1}
T_{\d\u}^\lambda = I + O(1/\lambda)$, by definition of $T^\lambda$.
Thus, (\ref{e:molamA}) may be written as
\begin{equation}
   \label{e:p2}
M_{0;\u}^\lambda(x) = (\sqrt\lambda (-\Pid)^{-1} + O(1)) N_0^\lambda(x),
\end{equation}
and similarly,
\begin{equation}
   \label{e:p3}
M_{b;\d}^\lambda(x) = (\sqrt\lambda (-\hPid)^{-1} + O(1)) N_b^\lambda(x).
\end{equation}
Equation (\ref{e:mofxb}) directly follows from (\ref{e:mofx},
\ref{e:p1}, \ref{e:p2}, \ref{e:p3}).
\end{proofbis}

It will be useful in Section~\ref{s:sd} to have separate expressions
for $\MM_0(x)$ and $\MM_b(x)$.  Using
\[
\vligne{I & e^{Kb} \\ e^{\widehat K b} & I}^{-1} =
\vligne{(I-e^{Kb} e^{\widehat K b})^{-1} & \\ & (I-e^{\widehat Kb} e^{K b})^{-1} }
\vligne{I & -e^{Kb} \\- e^{\widehat K b} & I},
\]
we easily replace (\ref{e:mofxb}) by the pair of equations
\begin{align}
   \label{e:mo}
\hspace*{-0.4cm} \MM_0(x) & = 2 (-P_b)^{-1} (I-e^{Kb} e^{\widehat K b})^{-1}  (\FF(K;x)
- e^{Kb} e^{\widehat K (b-x)} \FF(\widehat K; x)) \Delta_\sigma^{-1},
\\
   \label{e:mb}
\hspace*{-0.4cm} \MM_b(x) & = 2 (-\widehat P_b)^{-1} (I-e^{\widehat Kb} e^{K b})^{-1}
e^{\widehat K(b-x)}(\FF( \widehat K;x)
- e^{\widehat Kx} \FF(K; x)) \Delta_\sigma^{-1}.
\end{align}
The next corollary is
obvious: we merely let $x= b$ in (\ref{e:mo}, \ref{e:mb}).

\begin{cor}
   \label{t:phases}
If $\valpha \vmu \not= 0$, then the expected time spent in the various
phases during an excursion from one boundary to the other is given by
\begin{align*}
\MM_0(b) & = 2 (-P_b)^{-1} (I-e^{Kb} e^{\widehat K b})^{-1}  (\FF(K,b)
- e^{Kb} \FF(\widehat K; b)) \Delta_\sigma^{-1}
\\
\MM_b(b) & = 2 (-\widehat P_b)^{-1} (I-e^{\widehat Kb} e^{K b})^{-1}
(\FF( \widehat K,b) - e^{\widehat Kb} \FF(K; b)) \Delta_\sigma^{-1}.
\end{align*}
\qed
\end{cor}

\begin{rem} \em
In marked contrast to Theorem~\ref{t:probaup} and its
Corollary~\ref{t:probadown}, Theorem~\ref{t:expected} does not give an
expression for $\MM_0(x)$ and $\MM_b(x)$ if $\valpha \vmu = 0$.
In that case, we might write,
instead of (\ref{e:mofxb}), that $\MM(x)$ is to
be determined by solving the system
\[
\vligne{I  &  e^{Kb} \\  e^{\widehat Kb}  &  I}
\vligne{\Pid  \MM_0(x)  \\ \hPid \MM_b(x) }
= -2 \vligne{\FF(K;x) \\  e^{\widehat K(b-x)} \FF(\widehat K; x)} \Delta_\sigma^{-1},
\]
plus some additional equation.  Unfortunately, this additional
equation has eluded us so far.
\end{rem}

\section{Stationary distribution of a flexible MMBM}
\label{s:sd}

We have now obtained all the ingredients necessary to express the
stationary distribution of the flexible MMBM $\{Y(t), \rho(t)\}$ once
we specify its parameters.  

It is natural to expect {\em some} of the parameters at least to
take different values during the two legs of a regeneration cycle,
from level 0 to level $b$ and back.  We assume that the set $\EE$ of phases
is made up of two subsets, $\EE_u$ and $\EE_d$, and that the generator
of $\{\rho(t)\}$ is $Q$ partitioned as follows:
\begin{equation}
   \label{e:qstar}
Q = \vligne{Q_u & 0 \\ 0 & Q_d}.
\end{equation}
The idea is that $Q_u$, on the state space $\EE_u$, describes the
evolution of the Markov environment during the up-leg, after a
regeneration at level 0 until the next regeneration at level $b$;
$Q_d$ on the state space $\EE_d$ controls the system during a
down-leg, from $b$ to 0.

The other parameters are similarly partitioned and we write $\vmu =
\vligne{\vmu_u & \vmu_d}$ and $\vsigma = \vligne{\vsigma_u &
  \vsigma_d}$.  The matrices $P^\circ$ and $P^\bullet$ control the
transition from $\EE_d$ to $\EE_u$ upon hitting level 0 at the end of
a down-leg, and from $\EE_u$ to $\EE_d$ upon hitting $b$, and we write
them as
\[
P^\circ = \vligne{I & 0 \\ P_{du}^\circ & 0}
\qquad \mathrm{and} \qquad
P^\bullet = \vligne{0 & P_{ud}^\bullet \\ 0 & I}.
\]
The identity blocks on the diagonal do not play any role in the
calculation to follow, their role is to ensure that $P^\circ$ and
$P^\bullet$ are stochastic matrices.
We assume that $Q_u$ and $Q_d$ are irreducible, and so
Assumption~\ref{a:irreducible} is satisfied.

Upon hitting 0 at a regeneration point and after choosing a new phase
with the matrix $P^\circ$, the phase $\rho$ is in $\EE_u$.  Therefore,
the vector $\vnu_0$ takes the form $\vnu_0 = \vligne{\vnu_{0;u} &
  \vzero}$ and, for similar reasons, we have $\vnu_b = \vligne{\vzero
  & \vnu_{b;u}}$.

As transitions from $\EE_u$ to $\EE_d$ or from $\EE_d$ to $\EE_u$ are
possible at regeneration points only, the matrices of first passage
probabilities from one level to the other have the structure
\[
H_0 = \vligne{H_{0;u} & 0 \\ 0 & H_{0;d}}
\qquad \mathrm{and} \qquad
H_b = \vligne{H_{b;u} & 0 \\ 0 & H_{b;d}}
\]
The only blocks that one
needs to evaluate, however, are $H_{0;u}$ and $H_{b;d}$: the value of
$H_{0;d}$ is irrelevant as the process cannot leave level 0 in a phase
of $\EE_d$ and  $H_{b;u}$ is irrelevant as well, for a similar reason.

Obviously, the matrices $M_0(x)$ and $M_b(x)$ have the same structure
\[
M_0(x) = \vligne{M_{0;u}(x) & 0 \\ 0 & M_{0;d}(x)}
\qquad  \mathrm{and} \qquad
M_b(x) = \vligne{M_{b;u}(x) & 0 \\ 0 & M_{b;d}(x)},
\]
and we do not need to evaluate $M_{0;d}(x)$ or $M_{b;d}(x)$.

Altogether, we may re-formulate Theorem \ref{t:piofx} in a more
detailed manner as follows, using Theorem \ref{t:probaup}, Corollary
\ref{t:probadown}, and equations (\ref{e:mo}, \ref{e:mb}).

\begin{thm}
   \label{t:piofxB}
The stationary distribution $\Pi(x)$ of the flexible Markov modulated
Brownian motion is given by 
$
\Pi(x) = (\vnu^* \,   \vm^*)^{-1} \vnu^* M^*(x).
$
The vector $\vnu^*$ is partitioned as  $\vnu^* = \vligne{\vnu_{0;u} & \vnu_{b;d}}$ where 
\[
\vnu_{0;u} = \vnu_{0;u}   H_{0;u} P_{ud}^\bullet H_{b;d} P_{du}^\circ,  
\qquad 
\vnu_{b;d} = \vnu_{0;u} H_{0;u} P_{ud}^\bullet,
\]
with
\[
H_{0;u}  = \left.  \HH_0\right|_{Q=Q_u, \vmu = \vmu_u, \vsigma =
  \vsigma_u}
\qquad \mathrm{and} \qquad
H_{b;d}  = \left.  \HH_b\right|_{Q=Q_d, \vmu = \vmu_d, \vsigma =
  \vsigma_d}.
\]
The matrix $M^*(x)$ is partitioned as
\[
M^*(x)= \vligne{M_{0;u}(x) & \\ & M_{b;d}(x)},
\]
with
\[
M_{0;u}(x)  = \left. \MM_0(x)\right|_{Q=Q_u, \vmu = \vmu_u, \vsigma = \vsigma_u}
\   \mathrm{and} \ 
M_{b;d}(x) = \left. \MM_b(x)\right|_{Q=Q_d, \vmu = \vmu_d, \vsigma =
  \vsigma_d}.
\]
The vector $\vm^*$ is given by $\vm^* = M^*(b) \vone$.
\qed
\end{thm}

\section{Illustration}
\label{s:illustration}

\begin{exemple} \rm {\bf  Single-phase Brownian motion.}
   \label{ex:BM}
This is the example given in the introduction: the environmental
process has only one phase, possibly characterized by different
parameters in alternating intervals between regeneration points.  With
one phase only, the calculations simplify considerably; if $\mu$ is
negative, then $U= \widehat K=0$ and $\widehat U = K = 2
\mu/\sigma^2$.  Assuming that both $\mu_u$ and $\mu_d$ are negative, we obtain
from Theorem \ref{t:piofxB} that
\begin{align*} 
	\Pi(x)= \frac{\MM_{0;u}(x) + \MM_{b;d}(x)}{\MM_{0;u}(b) + \MM_{b;d}(b)},
\end{align*}
with
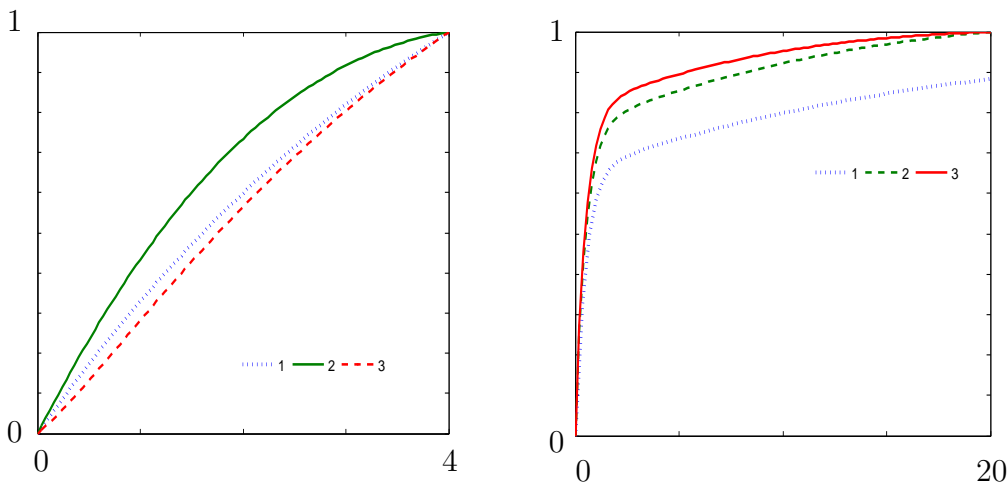
\begin{figure}[t]
\centering{
\includegraphics[width=0.4\textwidth,height=0.4\textwidth]{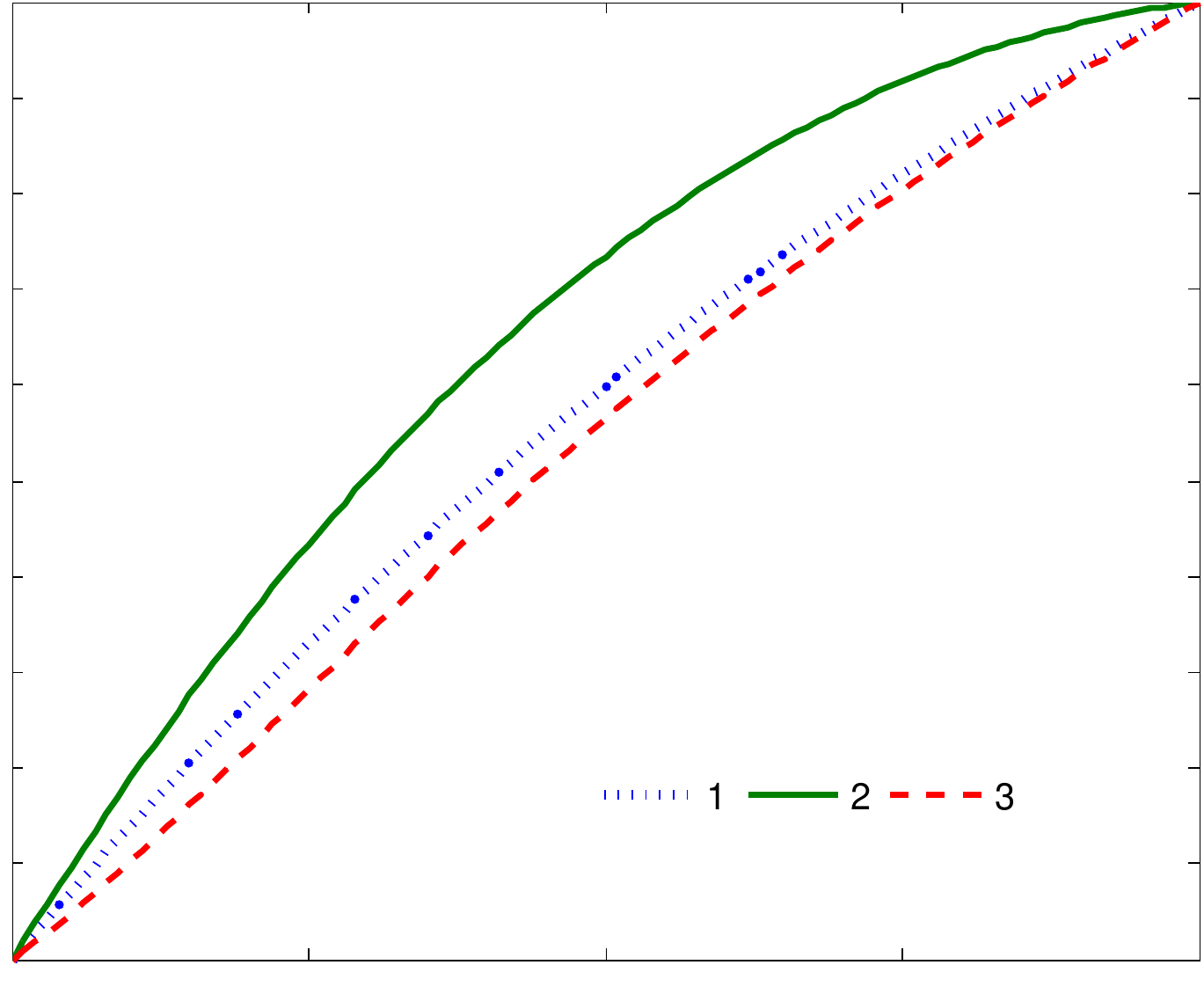} 
\put(0,-6){\makebox(0,0)[c]{$4$}}
\put(-153,-6){\makebox(0,0)[c]{$0$}}
\put(-163,5){\makebox(0,0)[c]{$0$}}
\put(-163,160){\makebox(0,0)[c]{$1$}}
\hspace{0.1\textwidth}
\raisebox{-0.01\textwidth}{%
\includegraphics[width=0.41\textwidth,height=0.41\textwidth]{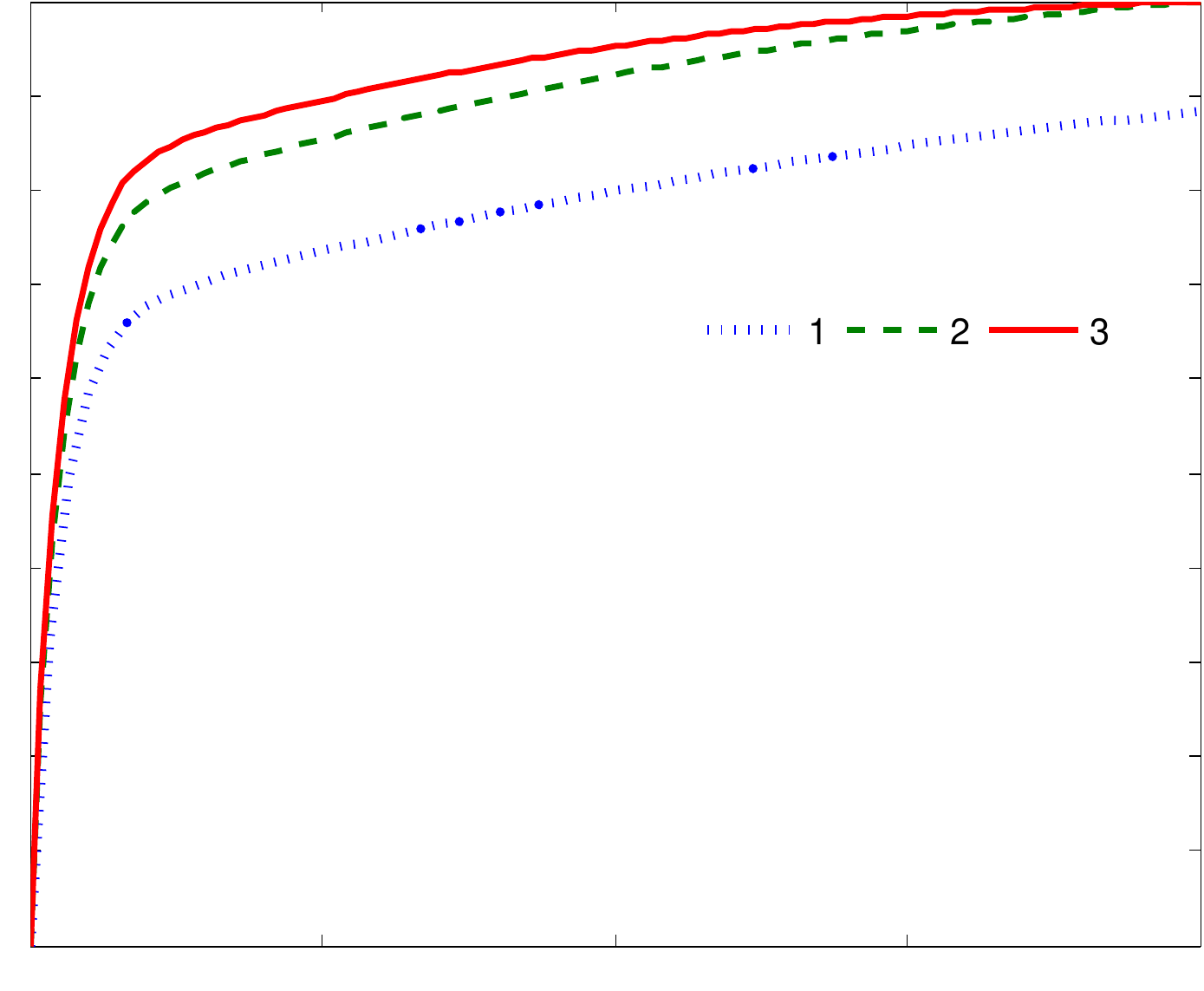} 
\put(0,-6){\makebox(0,0)[c]{$20$}}
\put(-153,-6){\makebox(0,0)[c]{$0$}}
\put(-163,5){\makebox(0,0)[c]{$0$}}
\put(-163,160){\makebox(0,0)[c]{$1$}}
}
\caption{\label{f:cdfcyclic}
Cumulative stationary distribution functions for the
single-phase BM examples (left) and the
cyclic-phases examples (right). The parameters are clarified in the text.}
}
\end{figure}
\[
\MM_{0;u}(x) + \MM_{b;d}(x) = \left(\frac{1}{\mu_d}-\frac{1}{\mu_u}\right) x
+\frac{\sigma_d^2}{2\mu_d^2} (1-e^{2 \mu_d x /\sigma_d^2}) 
- \frac{\sigma_u^2 }{2\mu_u^2} e^{-2 \mu_u b /\sigma_u^2} (1-e^{2 \mu_u x /\sigma_u^2}).
\]
If the parameters in the two types of intervals are equal, then
further simplifications yield the well-known truncated exponential
distribution 
	\begin{align*} 
		\Pi(x) = (1-e^{2 \mu x /\sigma^2}) (1-e^{2 \mu b /\sigma^2})^{-1}.
	\end{align*}

The parameters for the three distributions shown on the left of Figure \ref{f:cdfcyclic} are given in the table below.
\begin{center}
  \begin{tabular}{cc|cccc}
    Case & line &$\mu_u$ & $\sigma_u^2$ & $\mu_d$ & $\sigma_d^2$  \\
    \hline
  1 & dotted & -1 & 10 & -1 & 10  \\
 2 & plain  & -1 & 10 & -10 & 10  \\
 3 & dashed  & -1 & 10 & -1 & 1 
  \end{tabular}
\end{center}
In Case 1, with a single set of parameters, the process is a regulated
Brownian motion with two boundaries, in Case 2 the drift is decreased to $-10$ when the process reaches the upper boundary, and in Case 3
the drift remains the same, but the variance is reduced.

As expected, the buffer content is stochastically smaller in Case 2 than in Case 1:
for instance, the 90th percentiles are 2.88 and 3.44 respectively.  We
had expected that the buffer content would also be smaller in Case 3, our argument being that, with a smaller variance, the
negative drift would be better felt, and the buffer content would go
down faster.
As one sees on Figure \ref{f:cdfcyclic} this is not the case and the
buffer content is slightly larger in Case 3 (the 90th percentile is
3.48).   

We give in the table below the expected total duration of excursions
from level 0 to level $b$ and back from level $b$ to level 0, and also
the proportion of time spent by the process in the regenerative
intervals from 0 to $b$; this quantity is
$\Pi_u(b)= (\vnu^* \vm^*)^{-1} \vnu_{0;u} M_{0;u}(b) \vone$.

\begin{center}
  \begin{tabular}{c|ccc}
 Case &  $\MM_{0;u}(b)$ & $\MM_{b;d}(b)$ & $\Pi_u(b)$ \\
 \hline
    1  & 2.13 & 1.25 & 0.63  \\
    2  & 2.13 & 0.35 & 0.86 \\
    3  & 2.13 & 3.50 & 0.38
  \end{tabular}
\end{center}

Obviously, the time to move from 0 to $b$ is the same in all cases,
and we do observe for Case 2 the effect resulting from switching from
$\mu_u = -1$ to $\mu_d=-10$.  In Case 3, switching from
$\sigma_u^2 = 10$ to $\sigma_d^2 =1$ increases the expected length of
an excursion from
$b$ to 0 nearly by a factor 3.

\end{exemple}

\begin{exemple}   \rm {\bf Cyclic environmental process.}
   \label{ex:cyclic}
In this example, $m=8$ and the process of phases evolves
cyclically from 1 to 8 and back to 1.  We take $Q_u = Q_d = \Omega$ with
$\Omega_{i,i+1} = \lambda$, for $i=1, \ldots, 7$, $\Omega_{8,1}= \lambda$,
$\Omega_{ii}=-\lambda$ for all $i$, 
the other elements are equal to 0. In the three cases to
follow, we have $\lambda = 0.1$, so the process moves from one phase to the next in $10 $ units of time on average. The other parameters are
$\mu_i =-1$ for all $i$, and $\sigma_i= 1$ for all
$i\not=8$ and $\sigma_8=10$.  Thus, the process is quite regular most
of the time but every 80 units of time, on average, the volatility
becomes very high during 10 units of time.

In Case 1 (dotted line on the right-hand side graph of Figure
\ref{f:cdfcyclic}), the buffer is infinite.  One observes the effect of
the irregularity, infrequent but very high, of the input process: the
 stationary expected buffer occupancy is $\E[X_\infty] =
6.69$, but the distribution has a very long tail, with $\P[X_\infty >
20]=0.12$.  Actually, this tail decreases at a rate equal to the
maximal eigenvalue of  $K$, equal to $-0.14$ in the present example.

In Case 2 (dashed line), the buffer is finite, with $b=20$ and the
other parameters are the same as  in Case 1.  In Case 3 (plain line), $\vmu_u$
and $\vsigma_u= \vsigma_d$  are the same as in Case 2, and $\vmu_d = 10
\vmu_u$.   One clearly see that the buffer content is smallest in Case
3,  the 90th percentile, for instance, is 5.40, compared to 8.20 in
Case 2.   

It is interesting to examine in more details the behavior of the two
processes. The transition matrix $H_{0;u}$ is the same in both cases and the probability mass is almost exclusively concentrated
on the $8$th column: the computed values are $(H_{0;u})_{i,8} = 0.9995$ and
$(H_{0;u})_{i,1} = 0.0005$, independently of $i$, the remaining
elements of the matrix being negligible.

\begin{table}[t]
\begin{center}
  \begin{tabular}{l|rrrrrrrr}
 & 1~~ & 2~~ & 3~~ & 4~~ & 5~~ & 6~~ & 7~~ & 8~~ \\
\hline \hline
$\MM_{0;u}$ & 109.20  & 99.20 &  89.20 & 79.20 &  69.20 &  59.20 &  49.20 &  39.24 \\
\hline 
$\MM_{b;d}$ Case 2 &  19.49 &  19.45 &  19.34 &
   19.03 &  18.22 &  16.36 &  12.66 &   6.34 
\\
$\MM_{b;d}$ Case 3 &   2.00 &   2.00 &   2.00&
    2.00 &   2.00 &   1.99 &   1.97 &   1.53 
  \end{tabular}
\end{center}
\caption{
   \label{t:momentsCyclic}
Cyclic environment, moments of first passage time from one
boundary to the other.}
\end{table}

We give in Table~\ref{t:momentsCyclic} the expected duration of
transitions from one boundary to the other.  It appears clearly that
to reach level $b$, starting from a phase $i \not= 8$, the process must
first move to phase $8$, with an expected time equal to $10(8-i)$, and
only then get a chance to reach $b$ in a reasonable amount of time.
Furthermore, $(\MM_{0;u})_8$ is much greater than the expected sojourn
time in phase $8$.  We interpret this as follows: starting from level
0 in phase $8$, there is a significant probability that the process
reaches level $b$ before switching to phase 1, but it is also possible the system will have to go through one cycle (or more) before eventually reaching level $b$.

For Case 2, the effect of $\sigma_8$ is also noticeable in the
expected duration of excursions from $b$ to 0, albeit to a lesser
degree; for Case 3, these expected durations are dominated by the
large absolute value of $\vmu_d$.

\end{exemple}

\begin{exemple}   \rm{\bf Video streaming application.}
    \label{ex:telek}
This example is taken from Gribaudo {\it et al.}~\cite{gmst08}.  We
use it to illustrate changes in the system characteristics resulting
from global changes of the parameters.

There are five states; the video streaming application
cycles between {\em buffering} (States 1 and 3), {\em
  playing} (States 2 and 4), and {\em finishing} (State 5), leaving
each state at the rates $\beta_B$, $\beta_P$ and $\beta_F$, respectively. The videos are being played in a loop: when a video is finished, the application starts another one. Video streaming packets arrive at rate $\lambda_L$ with variance $\gamma_L$
when the network is congested (States 1 and 3), at rate $\lambda_H$ with variance
$\gamma_H$ otherwise (States 2 and 4). The packets are decoded at rate~$\delta$ with variance~$\gamma$. 

The transition matrix for the phase process is 
 \[
 Q = \vligne{
  \ast & \beta_B & \alpha_{LH} & 0 & 0 \\
  0 & \ast & 0 & \alpha_{LH} & \beta_P \\
 \alpha_{HL} & 0 & \ast & \beta_B & 0 \\
 0 & \alpha_{HL} & 0 & \ast & \beta_P \\
 \beta_F p_1 & 0 & \beta_F p_3 & 0 & \ast
 }
 \]
where the diagonal elements are such that $Q \vone = \vzero$, and the
rates and variance vectors are
\begin{align*}
\vmu & = \vligne{\lambda_L & \lambda_L-\delta & \lambda_H & \lambda_H - \delta
                                                       & - \delta},
\\
\vsigma^2 & = \vligne{\gamma_L & \gamma_L+\gamma & \gamma_H &
                                                               \gamma_H+\gamma
                                                       & \gamma}.
\end{align*}

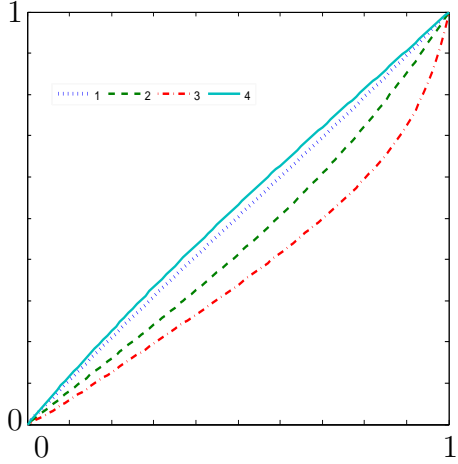
\begin{figure}[t]
\centering{
\includegraphics[width=0.41\textwidth,height=0.41\textwidth]{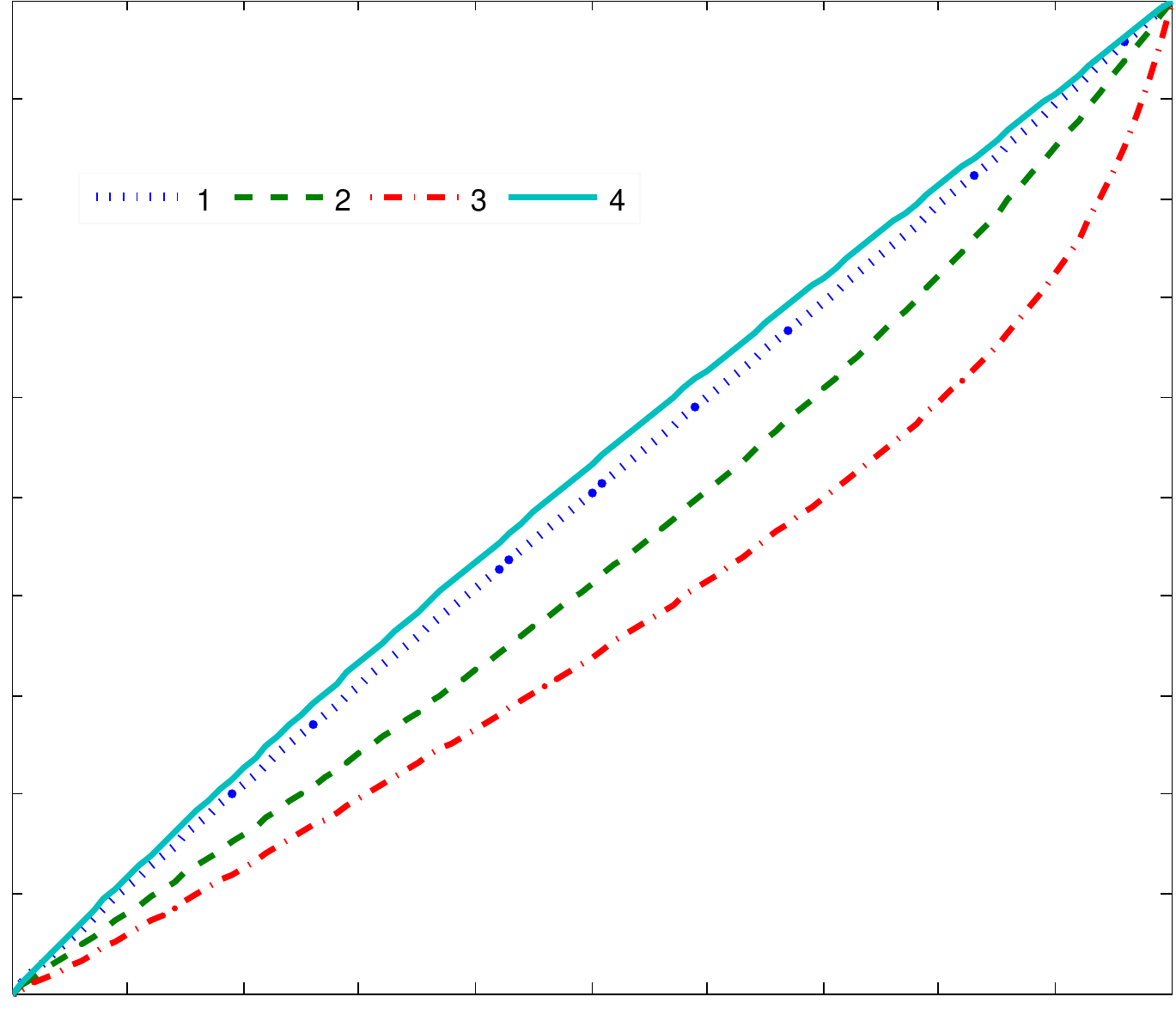} 
\put(0,-4){\makebox(0,0)[c]{$1$}}
\put(-153,-4){\makebox(0,0)[c]{$0$}}
\put(-163,6){\makebox(0,0)[c]{$0$}}
\put(-163,160){\makebox(0,0)[c]{$1$}}

\caption{\label{f:cdftelek}
Cumulative stationary distribution functions for Video Streaming Application}
}
\end{figure}

The distribution shown as a dotted line on
Figure~\ref{f:cdftelek} is that of the regulated MMBM, with parameters chosen
from~\cite{gmst08}: the buffer size is 1MB and is the unit of
volume, the time unit is 1 second, the parameters values are given in
the table below.
\begin{center}
  \begin{tabular}{lll}
    $\delta = \gamma = 0.5$ &   $\lambda_L = \gamma_L =0.25$ & $\lambda_H =
         \gamma_H =  0.625$ \\  
\hline
$\alpha_{LH} = \alpha_{HL} = 1/60$  & $\beta_B = \beta_F = 0.1$&
$\beta_P=0.03$
    \\
\hline
$p_1= \alpha_{HL} /(\alpha_{HL} + \alpha_{LH})$ & $p_3=1-p_1$
  \end{tabular}
\end{center}
The other curves are defined as follows:
\begin{itemize}
\item[] 
Case 2, dashed line --- the parameters are the same but the
phase transition matrices at regeneration epochs are $P^\circ =
P^\bullet =(1/m) \vone \, \vone^{\top}$; thus, the phase is sampled at random at the
end of each excursion from one boundary to the other.
\item[]
Case 3, mixed dashed and dotted line --- the variances are reduced
during excursions from $b$ to $0$, with $\vsigma_d = 0.1 \vsigma_u$;
all other parameters remain the same and $P^\circ = P^\bullet = I$.
\item[]
Case 4, plain line --- the system cycles 10 times faster through its
three stages during excursions from $b$ to 0, with $(\beta_P)_d =
(\beta_F)_d = 1$, $(\beta_P)_d=0.3$; all other parameters remain the
same and $P^\circ = P^\bullet = I$. 
\end{itemize} 
In addition to the distribution functions in Figure~\ref{f:cdftelek},
we give below the median of the four distributions, and also the
stationary probability $\Pi_u(b)$.  We observe that the buffer is more
heavily utilized when the variances are reduced (Case 3), that is, it spends half of the time being above level $0.7$. We also observe that re-sampling the phases at regeneration epochs (Case 2)
has a significant effect.

\begin{center}
  \begin{tabular}{c|cccc}
    Case &1& 2 &3& 4  \\
    \hline
Median & 0.50 & 0.59 & 0.70 & 0.46 \\
$\Pi_u(b)$ & 0.51 & 0.40 & 0.19 & 0.54
  \end{tabular}
\end{center}

\end{exemple}

\section{Comparison with existing literature}
\label{s:lothar}

A related question is addressed in Breuer~\cite{breue12}, where the
author analyzes the joint distribution of the random variables
\begin{align*}
\zeta_1(b;x,j) & = \int_0^{\delta_0 \wedge \delta_b} \indic\{X(s) < x,
\kappa(s)=j\} \ud s
\intertext{and}
\zeta_2(b;x,j) & = \int_0^{\delta_0 \wedge \delta_b} \indic\{X(s) > x,
\kappa(s)=j\} \ud s
\end{align*}
for $j \in \EE$ and $0 < x < b$.  These are the time spent in
$(0,x)\times \{j\}$ and $(x,b)\times \{j\}$ respectively, before the
first exit from the interval $(0,b)$.

We need to introduce some more notation.  Consider a vector $\vr
\geq \vzero$ indexed by $\EE$.  Define the matrix $U(\vr)$ as the minimal
solution of the matrix equation 
\begin{equation}
   \label{e:Uofr}
\Delta_\sigma^2 X^2 + 2 \Delta_\mu X + 2 (Q - \Delta_r) =0.
\end{equation}
For $\vr = \vzero$, we have $U(\vzero)=U$, the generator introduced at
the beginning of Section \ref{s:fluid}.  Similarly, $-\widehat U(\vr)$
is the maximal solution of (\ref{e:Uofr}).

We further define the random variables 
\[
T_1(b;x) = \sum_{j \in \EE} (\vr_1)_j \zeta_1(b;x,j),
\qquad
T_2(b;x) = \sum_{j \in \EE} (\vr_2)_j \zeta_2(b;x,j),
\]
where $\vr_1$ and $\vr_2$ are two nonnegative vectors.  The functions 
\begin{align*}
E _{\vr_1,\vr_2}^\u(b;x|a) = \E[ e^{-T_1(b;x) -T_2(b;x)} \indic\{\delta_b <
\delta_0\}, \kappa(\delta_b)| X(0)=a, \kappa(0)]
\intertext{and}
E _{\vr_1,\vr_2}^\d(b;x|a) = \E[ e^{-T_1(b;x) -T_2(b;x)} \indic\{\delta_0 <
\delta_b\}, \kappa(\delta_0)| X(0)=a, \kappa(0)],
\end{align*}
for $0 < a < b$, are the joint Laplace transforms of the $\zeta_1$s
and $\zeta_2$s
restricted on the exit occurring at the upper or lower boundary,
respectively, conditionally given that the process starts from level $a$ in the open interval $(0,b)$.

From \cite[Theorem 1, Lemmas 1 and 2]{breue12}, we find after
various adaptations to our specific case  and some simple manipulations that
\begin{align}
E _{\vr_1,\vr_2}^\u(b;x|x) & = -2 (\widehat P_x(\vr_1) + P_{b-x}(\vr_2))^{-1}
L_{b-x}(\vr_2)
\intertext{and, by symmetry, that}
E _{\vr_1,\vr_2}^\d(b;x|x) & = -2 (\widehat P_x(\vr_1) + P_{b-x}(\vr_2))^{-1} \widehat
L_{x}(\vr_1). 
\end{align}
Here, $\widehat P_x(\vr_1)$ and $ \widehat L_{x}(\vr_1)$ are given by
(\ref{e:Lhat}) with $U$, $\widehat U$ and $b$ respectively replaced by
$U(\vr_1)$, $\widehat U(\vr_1)$ and $x$, and $P_{b-x}(\vr_2)$ and
$L_{b-x}(\vr_2)$ are given by (\ref{e:LoPo}) with $U$, $\widehat U$
and $b$ replaced by $U(\vr_2)$, $\widehat U(\vr_2)$ and $b-x$.


Finally, we define the random variables
\[
\xi(y;x,j) = \int_0^{\delta_y} \indic\{Z(s)\leq x, \kappa(s)=j\}  \ud s,
\]
for $y \geq x$, as the time spent in $[0,x] \times \{j\}$ by the {\em
  regulated} process until the first passage to
level $y$, and we denote their joint Laplace transform, starting from
level 0, by
\[
\Xi_{\vr}(y;x) = \E[e^{-\vr \vxi(y;x)} \kappa(\delta_y) | Z(0=0, \kappa(0)].
\] 

\begin{lem}
   \label{t:generating}
The joint Laplace transform $\Xi_{\vr}(y;x)$ of the random variables
$\xi(b;x,j)$, $1 \leq j \leq m$, is given by
\begin{equation}
   \label{e:xi}
\Xi_{\vr} (b;x) = (I - \Xi_{\vr} (x;x) E_{{\vr},\vzero}^\d (b;x|x))^{-1}
\Xi_{\vr} (x;x) E_{{\vr},\vzero}^\u (b;x|x)
\end{equation}
for $b > x$, with  $\Xi_{\vr} (x;x)= (-P_x({\vr}))^{-1} L_x({\vr})$.
\end{lem}

\begin{proof}
We decompose the interval $[0, \delta_y]$ into three subintervals: 
\[
[0, \delta_y] = [0, \delta_x] \cup [\delta_x, \delta^*] \cup
[\delta^*, \delta_b],
\] 
where $\delta^* = \inf\{t > \delta_x : Z(t)=0 \
\mathrm{or} \ Z(t)=b\}$.  The sojourn times in $(0,x)\times\{j\}$
during these intervals are conditionally independent, given the
phases, and so we have 
\[
\Xi_{\vr} (b;x) = \Xi_{\vr} (x;x) (E_{{\vr},\vzero}^\u (b;x|x) +
E_{{\vr},\vzero}^\d (b;x|x) \Xi_{\vr} (b;x)),
\]
from which (\ref{e:xi}) follows.

The given expression for $\Xi_{\vr}(x;x)$ is a consequence of
\cite[Theorem 1]{breue12b}: we adapt it to our specific case, taking into account the fact that Theorem~1 in \cite{breue12b} is stated for the
level-reversed process, and performing some simple manipulations.
\end{proof}

At first, it looks like we might obtain $\MM_0(x)$ by differentiating
both sides of (\ref{e:xi}) with respect to $\vr$ and by evaluating the
result at $\vr = \vzero$.  We would need the derivatives with respect
to $\vr$ of the solutions of (\ref{e:Uofr}).  Details would still need
to be worked out and in final analysis, the expressions so obtained
would without doubt be much more involved than the very
clean expressions given in~(\ref{e:mofxb}).

\section{Conclusion and extensions}
\label{s:conclusion}

In this paper, we have illustrated one useful reason for approximating
Markov-modulated Brownian motions with stochastic fluid queues. In
particular, the approximation allows for the analysis of MMBMs subject
to boundary conditions that are not the traditional regulation. This
approach, coupled with the regenerative method, may be adapted easily
to other types of feedbacks, such as a combination of absorption,
stickiness, and instantaneous change of phase whenever the process
hits a boundary. 

With the technique developed here, we might analyze systems for which
the so-called feedback only lasts for a finite amount of time.  For
instance, rates change for an exponential amount of time, and then
the system resumes its normal mode of operations.  The results from
Sections \ref{s:proba} to \ref{s:sd} have to be adapted, as the
generator for the phase process between two regeneration points is no
longer irreducible.  


\appendix

\section{Expected time under a taboo}
   \label{s:gamma0}

   Although the statement of Theorem \ref{theo:GGhat} would seem to be
   well-known, we do not know of a published formal proof, which is
   why we include it here.  Consider a fluid queue with generator $T$
   for the environmental Markov process and fluid growth rates $\vcc$.
   We denote by $\Psi$ its matrix of first return probabilities to
   level 0, and we define
   $K = C_\u^{-1} T_{\u\u} + \Psi |C_\d|^{-1} T_{\u\d}$, where
   $C = \Delta_c$.

\begin{theo} 
		\label{theo:GGhat}
The matrix $\Gamma(x)$ of mean sojourn time in $[0,x]$ before return
to the initial level 0, starting from a phase with positive growth, is
given by 
	\begin{align} 
		\label{eqn:G} 
	\Gamma(x) & = \FF(K;x) \vligne{ C_\u^{-1} & \Psi |C_\d|^{-1}},
	\end{align} 
where $\FF(K;x)= \int_0^x e^{Ku} \ud u$.
\end{theo}

\begin{proof} 

We need to consider separately the case when the stationary drift is
strictly negative from the case when it is positive or equal to zero.

\paragraph{A. Strictly negative drift}
\label{sec:a}

In this case, the fluid queue is positive recurrent and the
eigenvalues of $K$ are all in $\C_{<0}$.

We define the complementary probability functions 
	\begin{align*}
	[G(x,t)]_{i j} & = \P[t < \tau, {X}(t) > x, \kappa(t) = j \, |\, {X}(0) = 0, \kappa(0) = i],  
	\end{align*}
	where $\tau$ is the first return time to level $0$. Denote by $\overline{\Gamma}(x)$ the mean sojourn time in $(x,\infty)$ before returning to the initial level $0$: 
	\begin{align*}
	\overline{\Gamma}(x) = \int_0^{\infty} G(x,u) \ud u,
	\end{align*} 
obviously, $\Gamma(x) = \overline{\Gamma}(0) - \overline{\Gamma}(x)$.
 One verifies by the usual argument (Karandikar and
 Kulkarni~\cite{kk95})  that $G(x,t)$ is the solution to the system of partial differential equations 
	\begin{align}
	\frac{\partial}{\partial t} G(x,t) +  \frac{\partial}{\partial x} G(x,t) C & = G(x,t) T \quad \mbox{ for } x > 0.  \nonumber
	\intertext{Integrating both sides with respect to $t$ from $0$ to $\infty$ gives} 
	%
	\label{eqn:pde}
[G(x,t)]_0^\infty + \frac{\partial}{\partial x} \overline \Gamma(x) C & = \overline \Gamma(x) T.
	\end{align} 
	Note that 
	\begin{align*} 
		\lim_{t \rightarrow \infty} G(x,t) =0, \qquad \lim_{t
  \rightarrow 0} G(x,t) = 0, 
	\end{align*} 
	the first limit is due to the negative drift assumption, which
        implies that $\tau < \infty$ almost surely, the second holds
        because the fluid queue does not have enough time to grow
        beyond~$x$ by time $t$ if $t$ is small. Then, it is easy to
        verify that the solution to~\eqref{eqn:pde} is given by
\begin{equation}
\overline \Gamma(x) = A K^{-1} e^{Kx} \vligne{C_\u^{-1}  &  \Psi |C_\d|^{-1}}
\end{equation}
for some matrix $A$ to be determined.

Let us focus on the block $\overline \Gamma_{\u\u} (x) = A K^{-1}
e^{Kx} C_\u^{-1}$.  For sufficiently small
$h$, we may write that
\begin{equation}
   \label{e:Bofx}
\overline \Gamma_{\u\u}(0) = h C_\u^{-1} + \overline \Gamma_{\u\u}(h) + o(h).
\end{equation}
Indeed, the expected time spent above level 0 in a phase $j$ is equal
to the time needed to reach level $h$ in that phase, if $\kappa(0)=j$
plus the time spent above $h$ in that phase.  The third term in
(\ref{e:Bofx}) accounts for the time spent in oscillations between 0
and $h$ whenever the fluid drops below $h$.  We get from
(\ref{e:Bofx}) that
	\begin{align*} 
\frac{\partial}{\partial x} \overline \Gamma_{\u\u}(x)|_{x=0} = -C_\u^{-1},
	\end{align*} 
from which we conclude that $A= -I$.  Thus,
\begin{align*}
\Gamma(x) & = (-K)^{-1} (I- e^{Kx}) \vligne{ C_\u^{-1} & \Psi
  |C_\d|^{-1}} \\ 
  & = \FF(K;x) \vligne{ C_\u^{-1} & \Psi   |C_\d|^{-1}} \qquad \qquad
  \mbox{by Lemma \ref{t:F}.} 
\end{align*}

\paragraph{B. Nonnegative drift}
\label{sec:b}

In this case, the fluid queue is null-recurrent or transient, and $K$
has one eigenvalue equal to zero.

We may not repeat the argument for Part \ref{sec:a}  because the mean
sojourn time in $(x,\infty)$ is infinite, and $\int_0^{\infty}
G(x,u)\ud u$ diverges for any $x$.  
		
To get around this problem, we shall kill the process after a random,
finite, interval of time and then use a limiting argument. We define
$\zeta$ to be an exponentially distributed random variable with rate
$\psi$, and
		\begin{align*} 
	G'(x,t; \psi) & = \P[t < \min\{\tau, \zeta\}, {X}(t) > x, \kappa(t)  \; |\; {X}(0) = 0, \kappa(0) ], \\
					\Gamma'(x;\psi) & = \int_0^{\infty} G'(x,u;\psi) \ud u. 
		\end{align*} 
		
As $\zeta < \infty$ with probability one, $\Gamma'(x;\psi) < \infty$
and we may retrace the steps in Part \ref{sec:a}. In particular, the
matrix $\Gamma'(x;\psi)$ of expected sojourn time in $(x,\infty)$
under the taboo of $0$ is given by  
\begin{align*}
\Gamma'(x;\psi) = (-{K}_\psi)^{-1}e^{{K}_{\psi} x} \vligne{ C_\u^{-1} & \Psi_\psi |C_\d|^{-1}},
\end{align*}  
where ${\Psi}_{\psi}$ is the minimal nonnegative solution of the Riccati equation 
		\begin{align*}
			|C_\d^{-1}|T^{\psi}_{-+} + {\Psi}_{\psi}C_\u^{-1}T^{\psi}_{++} + |C_\d^{-1}|T^{\psi}_{--} {\Psi}_{\psi} + {\Psi}_{\psi} C_\u^{-1}T^{\psi}_{+ -} {\Psi}_{\psi} = 0, 
		\end{align*} 
		with 
		\begin{align*} 
		T^{\psi} = \left[\begin{array}{cc} T_{++} - \psi I & T_{+-} \\ 
		T_{-+}  & T_{--} - \psi I \end{array} \right],
		\end{align*}
		and 
		\begin{align*}
		{K}_{\psi} & =  |C_\d^{-1}|T_{--}^{\psi} + {\Psi}_{\psi}C_\u^{-1}T_{+-}^{\psi}.
		\end{align*}
In the limit as $\psi \rightarrow 0$, the matrices $\Psi_{\psi}$ and ${K}_{\psi}$ respectively converge to ${\Psi}$ and ${K}$.
Then, 
\begin{align*} 
{\Gamma}(x) & = \lim_{\psi \rightarrow 0} (\Gamma'(0;
                \psi) - \Gamma'(x;\psi)) \\ 
 & = \lim_{\psi \rightarrow 0} \left \{(-{K}_{\psi})^{-1} (I -
   e^{{K}_{\psi}b}) \right\}  \vligne{C_\u^{-1} & \Psi |C_\d|^{-1}}. 
		\end{align*}
In the last expression, ${K}_{\psi}$ converges to ${K}$ which is singular; thus, we need to exercise some care in evaluating the remaining limit. 
		
The matrix ${K}$ has one isolated eigenvalue equal to $0$, and ${K}_{\psi}$ has an isolated, real, maximal eigenvalue $\omega_{\psi}$ which converges to $0$ as $\psi \rightarrow 0$. Thus, there exist some matrices $S$ and $S_{\psi}$ such that 
\begin{align*} 
{K} = S \left[\begin{array}{cc} J & \\ & 0 \end{array}\right] S^{-1} \quad \mbox{ and } \quad {K}_{\psi} = S_{\psi} \left[\begin{array}{cc} J_{\psi} & \\ & \omega_{\psi} \end{array} \right] S_{\psi}^{-1},
		\end{align*}
where $J$ is a matrix with all eigenvalues of ${K}$ in $\C_{<0}$,
$J_{\psi}$ is  a matrix with all eigenvalues of ${K}_{\psi}$ with real
parts strictly less than $\omega_\psi$, and $S_{\psi} \rightarrow S$ and $J_{\psi} \rightarrow J$ as $\psi \rightarrow 0$. We decompose the inverse of ${K}_{\psi}$ as 
		\begin{align}
			\label{eqn:kstarinv} 
		{K}_{\psi}^{-1} = S_{\psi} \left[\begin{array}{cc} J_{\psi}^{-1} & \\ & 0 \end{array}\right] S_{\psi}^{-1} + S_{\psi} \left[\begin{array}{cc} 0 & \\ & \omega_{\psi}^{-1}\end{array}\right]S_{\psi}^{-1}. 
		\end{align}
		Note that the first term on the right side of~\eqref{eqn:kstarinv} converges to the group inverse ${K}^{\#}$ as $\psi \rightarrow 0$ \cite[Theorem 7.2.1]{cm91}. Also, 
		\begin{align*}
		\lim_{ \psi \rightarrow 0} e^{{K}_{\psi} x} & = \lim_{\psi \rightarrow 0} S_{\psi}\left[\begin{array}{cc} e^{J_{\psi} x} & \\ & e^{\omega_{\psi}x} \end{array}\right] S_{\psi}^{-1} = S\left[\begin{array}{cc} e^{Jx} & \\ & 1 \end{array}\right]S^{-1}  = e^{{K}x}.
		\end{align*} 
Thus,
\begin{align*}
 \lim_{\psi \rightarrow 0} & (-{K}_{\psi})^{-1} (I - e^{{K}_{\psi}x})
 \\
& = \lim_{\psi \rightarrow 0}(-{K}_{\psi})^{-1}S_{\psi}
\left[\begin{array}{cc} I - e^{J_{\psi}x} &\\ & 1 - e^{\omega_{\psi}
      x}  \end{array}\right]S_{\psi}^{-1} 
\\
& = -{K}^{\#}(I - e^{{K}x}) - \lim_{\psi \rightarrow
                   0} S_{\psi}\left[\begin{array}{cc} 0 & \\ &
                     \omega^{-1}_{\psi} (1-e^{\omega_\phi
                       x)}\end{array}\right] S_{\psi}^{-1} 
\\
& = -{K}^{\#}(I- e^{{K}x}) + x \boldsymbol{v} \boldsymbol{u},
		\end{align*} 
		where $\boldsymbol{v}$ and $\boldsymbol{u}$ are
                respectively the right and left eigenvectors of ${K}$
                for the eigenvalue $0$.   By Lemma \ref{t:F} again, this completes the proof of~\eqref{eqn:G}.
			\end{proof}


\subsubsection*{Acknowledgements}
The authors thank the Minist\`ere de la Communaut\'e fran\c{c}aise de
Belgique for supporting this research through the ARC grant
AUWB-08/13--ULB~5, they acknowledge the financial support of the
Australian Research Council through the Discovery Grant
DP110101663. Giang Nguyen also acknowledges the support of ACEMS (ARC Centre
of Excellence for Mathematical and Statistical Frontiers).

\bibliographystyle{abbrv}
\bibliography{twoBfluid}	

\end{document}